\documentclass[a4paper]{extarticle}
\usepackage{amsmath,amssymb,amsthm,mathrsfs,mathtools}
\usepackage[left=2cm,right=3cm,top=2cm,bottom=2cm]{geometry}
\usepackage{comment}
\usepackage{url}
\usepackage{verbatim}
\usepackage{ytableau}
\usepackage[vcentermath]{youngtab}
\usepackage{tikz}
\usepackage{float}
\usetikzlibrary{arrows,arrows.meta,calc}
\usepackage{hyperref}
\usepackage{bm}

\definecolor{corlinks}{RGB}{0,150,0}
\definecolor{cormenu}{RGB}{0,0,150}
\definecolor{corurl}{RGB}{0,0,150}
\hypersetup{
	colorlinks=true,
	urlcolor=corlinks,
	linkcolor=corlinks,
	menucolor=cormenu,
	citecolor=corlinks,
	pdfborder= 0 0 0
}

\newcommand{\aS}{\mathfrak{S}}
\newcommand{\IN}{\mathbb{N}}

\newcommand{\IF}{\mathbb{F}}
\newcommand{\oIF}{{\overline{\mathbb{F}}}}
\newcommand{\IC}{\mathbb{C}}

\newcommand{\GL}{\textsf{\textup{GL}}}
\newcommand{\SL}{\textsf{\textup{SL}}}
\newcommand{\sgn}{\textup{sgn}}

\newcommand{\Wedge}{{\textstyle\bigwedge}}
\newcommand{\tensor}{{\textstyle\bigotimes}}
\newcommand{\Sym}{\textup{Sym}}
\newcommand{\sym}{\textup{sym}}
\newcommand{\la}{\lambda}
\newcommand{\supp}{\textup{supp}}
\renewcommand{\c}{\textup{\textsf{c}}}

\newcommand{\id}{\textup{id}}

\newcommand{\stab}{\textup{stab}}
\newcommand{\Hodge}{\textsf{\textup{D}}}
\newcommand{\Wronski}{\textsf{\textup{W}}}
\newcommand{\Hermite}{\textsf{\textup{R}}}
\newcommand{\KP}{\textsf{\textup{K}}}
\newcommand{\I}{\textsf{\textup{I}}}

\newcommand{\Diag}{\textsf{Diag}}
\newcommand{\adj}{\mathrm{adj}}
\renewcommand{\det}{\mathrm{det}}
\renewcommand{\k}{\textsf{\textup{k}}}
\renewcommand{\Im}{\mathsf{Im}}
\renewcommand{\d}{\mathsf{d}}
\newcommand{\s}{\mathsf{s}}
\newcommand{\ann}{\mathrm{ann}}
\newcommand{\coinv}{V_Z}
\newcommand{\Hom}{\textup{Hom}}

\swapnumbers
\newtheorem{theorem}{Theorem}
\numberwithin{theorem}{section}
\newtheorem{lemma}[theorem]{Lemma}
\newtheorem{proposition}[theorem]{Proposition}
\newtheorem{corollary}[theorem]{Corollary}
\newtheorem{claim}[theorem]{Claim}
\theoremstyle{definition}

\newtheorem{definition/}[theorem]{Definition}
\newenvironment{definition}
  {
   \pushQED{\qed}\begin{definition/}}
  {\popQED\end{definition/}}

\newtheorem{remark/}[theorem]{Remark}
\newenvironment{remark}
  {
   \pushQED{\qed}\begin{remark/}}
  {\popQED\end{remark/}}

\title{Field-independent Kronecker-plethysm isomorphisms}
\author{Christian Ikenmeyer\\{\footnotesize University of Warwick}\\{\footnotesize\url{christian.ikenmeyer@warwick.ac.uk}} \and Heidi Omar\\{\footnotesize University of Warwick}\\{\footnotesize\url{heidiomar.ho@gmail.com}} \and Dimitrios Tsintsilidas\\{\footnotesize University of Warwick}\\{\footnotesize\url{Dimitrios.Tsintsilidas@warwick.ac.uk}}}

\begin{document}
\raggedbottom

\maketitle

\begin{abstract}
We construct an explicit field-independent SL$_2$-equivariant isomorphism between an invariant space of tensors and a plethysm space.
The existence of such an isomorphism was only known in characteristic 0, and only indirectly via character theory.
Our isomorphism naturally extends the web of field-independent isomorphisms given by Hermite reciprocity, Hodge duality, and the Wronskian isomorphism.
This is a characteristic free generalization of a classical situation in characteristic zero: certain rectangular Kronecker coefficients coincide with certain plethysm coefficients, and their non-negativity proves the unimodality of the $q$-binomial coefficient.
In the dual situation, we also establish a field-independent version of a coinvariant space and show the corresponding isomorphisms.

We also give a short combinatorial field-independent proof of the known fact that the Hermite reciprocity map over the standard basis is a triangular matrix with 1s on the main diagonal.
\end{abstract}

\newcommand\blfootnote[1]{
  \begingroup
  \renewcommand\thefootnote{}\footnotetext{#1}
  \addtocounter{footnote}{-1}
  \endgroup
}
\blfootnote{For the purpose of open access, the authors have applied a Creative Commons Attribution (CC-BY) license to any Author Accepted Manuscript version arising from this submission.}

{\noindent\footnotesize\textbf{Keywords:} Representation theory of SL$_2$, Kronecker coefficients, plethysm coefficients, modular representation theory}

\section{Introduction}
The classical Hermite reciprocity law states that $\Sym^m(\Sym^\ell \IC^2)$ and $\Sym^\ell(\Sym^m \IC^2)$ are isomorphic $\GL_2(\IC)$-representations, where $\Sym^m W$ denotes the symmetric power. This statement is false over arbitrary fields, as was shown in \cite{Kou:90}. In fact, proper duals have to be taken, see \eqref{eq:Hermite} below.
 
For any field $\IF$, let $W$ be a $\GL_2(\IF)$-representation, and let $\Sym^m W = \big((\tensor^\bullet W)/\langle x\otimes y - y \otimes x\rangle\big)_m$ denote its $m$-th symmetric power. Let $\aS_m$ denote the symmetric group on $m$ symbols. Let $\Sym_m W =(\tensor^m W)^{\aS_m}$ denote the $m$-th divided power of $W$, i.e., the vector space of $\aS_m$-invariant order $m$ tensors.
The representations $\Sym^m W$ and $\Sym_m W$ are dual to each other (see Lemma~\ref{lem:duality} below).
The self-duality of $\SL_2(\IC)$-representations makes $\Sym^m W$ and $\Sym_m W$ isomorphic $\SL_2(\IC)$-representations over $\IC$, which makes it difficult to see the field-independent structure.
One gets a Hermite reciprocity isomorphism $\Hermite_{m,\ell}$ in a field-independent way after choosing the duals correctly:
\begin{equation}\label{eq:Hermite}
\Hermite_{m,\ell} : \Sym_m \Sym^\ell \IF^2 \ \stackrel{\sim}{\longrightarrow} \ \Sym^\ell\Sym_m\IF^2,
\end{equation}
see \cite{AFPRW:19} and \cite{MW:22}, later again discussed in \cite{RS:21}.
The isomorphism $\Hermite_{m,\ell}$ is defined via field-independent versions of classical isomorphisms, as depicted in Figure~\ref{fig:square}:
$\Hermite_{m,\ell} = \Wronski^\star_{\ell,m} \circ \Hodge_{m,\ell} \circ\Wronski_{m,\ell}$.
In this paper, we add new explicit isomorphisms to this picture that connect plethysm spaces (i.e., compositions of $\Sym^\bullet$ or $\Sym_\bullet$, such as $\Sym_m\Sym^\ell\IF^2$) with invariant and coinvariant spaces of tensors as follows.

\begin{figure}
\centering
\scalebox{1}{
\begin{tikzpicture}[xscale=3,yscale=0.8]
\node[anchor=west] at (-0.75,3) {The isomorphisms:};
\begin{scope}
\node (ll) at (0,0) {$\Sym_m \Sym^{\ell} \IF^2$};
\node (ul) at (0,2) {$\Wedge^m \Sym^{\ell+m-1} \IF^2$};
\node (lr) at (3,0) {$\Sym^\ell \Sym_m \IF^2$};
\node (ur) at (3,2) {$\Wedge^\ell \Sym_{\ell+m-1} \IF^2$};
\draw[-Triangle] (ll) -- (ul) node [midway, right, fill=white] {\footnotesize $\Wronski_{m,\ell}$};
\draw[-Triangle] (ul) -- (ur) node [midway, above, fill=white] {\footnotesize $\Hodge_{m,\ell} = \Hodge_{\ell,m}^\star$};
\draw[-Triangle] (ur) -- (lr) node [midway, right, fill=white] {\footnotesize $\Wronski_{\ell,m}^\star$};
\draw[-Triangle] (ll) -- (lr) node [midway, above, fill=white] {\footnotesize $\Hermite_{m,\ell}=\Hermite^{\star}_{\ell,m}$};
\node (kronI) at (0,-2) {$\big(\Sym_{\ell m}(\IF^{\ell\times\ell\times2})\big)_{\SL_\ell(\oIF)\times \SL_\ell(\oIF)}$};
\node (kronIID) at (3,-2) {$\big(\Sym^{\ell m}(\IF^{m\times m\times 2} )\big)^{\SL_m(\oIF)\times \SL_m(\oIF)}$};
\draw[-Triangle] (kronIID) -- (lr) node [midway, right, fill=white] {\footnotesize $\KP_{\ell,m}$};
\draw[Triangle-] (kronI) -- (ll) node [midway, right, fill=white] {\footnotesize $\KP^\star_{m,\ell}$};
\draw[-Triangle] (kronI) -- (kronIID) node [midway, above, fill=white] {\footnotesize $\I_{m,\ell}=\I_{\ell,m}^{\star}$};
\end{scope}
\node[anchor=west] at (-0.75,-3.5) {The dual situation. The same maps appear, just with swapped parameters:};
\begin{scope}[shift={(0,-6.5)}]
\node (lld) at (0,0) {$\Sym^m \Sym_{\ell} \IF^2$};
\node (uld) at (0,2) {$\Wedge^m \Sym_{\ell+m-1} \IF^2$};
\node (lrd) at (3,0) {$\Sym_\ell \Sym^m \IF^2$};
\node (urd) at (3,2) {$\Wedge^\ell \Sym^{\ell+m-1} \IF^2$};
\draw[Triangle-] (lld) -- (uld) node [midway, right, fill=white] {\footnotesize $\Wronski^\star_{m,\ell}$};
\draw[Triangle-] (uld) -- (urd) node [midway, above, fill=white] {\footnotesize $\Hodge_{\ell,m} = \Hodge_{m,\ell}^\star$};
\draw[Triangle-] (urd) -- (lrd) node [midway, right, fill=white] {\footnotesize $\Wronski_{\ell,m}$};
\draw[Triangle-] (lld) -- (lrd) node [midway, above, fill=white] {\footnotesize $\Hermite_{\ell,m}=\Hermite^\star_{m,\ell}$};
\node (kronId) at (0,-2) {$\big(\Sym^{\ell m}(\IF^{\ell\times \ell\times2})\big)^{\SL_\ell(\oIF)\times \SL_\ell(\oIF)}$};
\node (kronII) at (3,-2) {$\big(\Sym_{\ell m}(\IF^{m\times m\times 2})\big)_{\SL_m(\oIF)\times \SL_m(\oIF)}$};
\draw[Triangle-] (kronII) -- (lrd) node [midway, right, fill=white] {\footnotesize $\KP_{\ell,m}^\star$};
\draw[-Triangle]  (kronId) -- (lld) node [midway, right, fill=white] {\footnotesize $\KP_{m,\ell}$};
\draw[Triangle-] (kronId) -- (kronII) node [midway, above, fill=white] {\footnotesize $\I_{\ell,m}=\I_{m,\ell}^{\star}$};
\end{scope}
\end{tikzpicture}
}
\caption{Commutative diagrams of field-independent equivariant isomorphisms and their duals. The map $\Hermite_{m,\ell}$ is Hermite reciprocity, $\Wronski_{m,\ell}$ is the Wronskian isomorphism, $\Hodge_{m,\ell}$ is the Hodge duality, $\KP_{m,\ell}$ is our Kronecker-plethysm isomorphism, and $\I_{m,\ell}$ is our isomorphism between coinvariant and invariant spaces.}
\label{fig:square}
\end{figure}

Let $\oIF$ denote the algebraic closure of~$\IF$, and fix an embedding $\IF\subseteq\oIF$.
We have $\IF^n\subseteq\oIF^n$, $\Sym^m \IF^n \subseteq \Sym^m \oIF^n$, etc.
The group $G_\ell := \GL_\ell(\oIF)\times\GL_\ell(\oIF)\times\GL_2(\oIF)$ acts on the tensor product
$\oIF^{\ell\times\ell\times 2}:=\oIF^\ell \otimes \oIF^\ell \otimes \oIF^2$
via $(g_1,g_2,g_3)(v_1\otimes v_2 \otimes v_3) = g_1(v_1) \otimes g_2(v_2) \otimes g_3(v_3)$ and extended $\oIF$-linearly.
This action lifts to the tensor algebra $\tensor^\bullet(\oIF^{\ell\times\ell\times 2})$
via $g(w_1\otimes w_2 \otimes \cdots \otimes w_d) := (g w_1)\otimes \cdots \otimes (g w_d)$ and extended $\oIF$-linearly,
for $g\in G_\ell$.
This action induces a linear action of $G_\ell$ on $\Sym^\bullet(\oIF^{\ell\times\ell\times 2})$
and on every $\Sym_m(\oIF^{\ell\times\ell\times 2})$.

For any $G_\ell$-representation $W$, let 
$W^{\SL_\ell(\oIF)\times\SL_\ell(\oIF)}$
denote the space of invariants under the action of the group $\SL_\ell(\oIF)\times\SL_\ell(\oIF)$,
interpreted as a subgroup of $G_\ell$ via the embedding $(g_1,g_2) \mapsto (g_1,g_2,\id_{\oIF^2})$.
Embed $\GL_2(\IF)$ into $G_\ell$ via $g\mapsto (\id_{\oIF^\ell},\id_{\oIF^\ell},g)$ and 
define the $\GL_2(\IF)$-representation
\begin{equation}\label{eq:invariants}
(\Sym^\bullet(\IF^{\ell \times\ell \times 2}) )^{\SL_\ell(\oIF)\times\SL_\ell(\oIF)}
\, \ := \ \,
(\Sym^\bullet(\oIF^{\ell\times\ell\times 2}))^{\SL_\ell(\oIF)\times\SL_\ell(\oIF)}
\ \cap \ 
\Sym^\bullet(\IF^{\ell \times\ell \times 2}).
\end{equation}
This is a slightly unusual notation, because the group $\SL_\ell(\oIF)\times\SL_\ell(\oIF)$ does not act on $\Sym^\bullet(\IF^{\ell \times\ell \times 2})$, only on $\Sym^\bullet(\oIF^{\ell \times\ell \times 2})$. But it turns out that this definition lets us
 
establish explicitly the Kronecker-plethysm $\GL_2(\IF)$-isomorphism
\[
\KP_{m,\ell} : \big(\Sym^{\ell m}(\IF^{\ell \times\ell \times 2})\big)^{\SL_\ell(\oIF)\times \SL_\ell(\oIF)}        \ \stackrel{\sim}{\longrightarrow} \      \Sym^m \Sym_\ell \IF^2   .
\]
Previously, the existence of such an isomorphism was only known in characteristic zero, and only via the character theory of the symmetric group, see \cite{PP:13,PP:14}.
We show in \S\ref{sec:computations} via a direct calculation that the algebraic closure in the definition is necessary,
i.e., we give examples where
$\big(\Sym^{\ell m}(\IF^{\ell \times\ell \times 2})\big)^{\SL_\ell(\IF)\times \SL_\ell(\IF)} \ \not\sim \ \Sym^m \Sym_\ell \IF^2$.
 
We also define the dual of the isomorphism $\KP_{m,\ell}$:
\[
\KP_{m,\ell}^\star :    \Sym_m \Sym^\ell \IF^2      \ \stackrel{\sim}{\longrightarrow} \ \big(\Sym_{\ell m}(\IF^{\ell \times\ell \times 2})\big)_{\SL_\ell(\oIF)\times \SL_\ell(\oIF)}    \ ,
\]
where
$\big(\Sym_{\ell m}(\IF^{\ell \times\ell \times 2})\big)_{\SL_\ell(\oIF)\times \SL_\ell(\oIF)}$
is defined as the set of those equivalence classes in
$\big(\Sym_{\ell m}(\oIF^{\ell \times\ell \times 2})\big)_{\SL_\ell(\oIF)\times \SL_\ell(\oIF)}$
that contain a representative in $\Sym_{\ell m}(\IF^{\ell \times\ell \times 2})$, see \S\ref{sec:coinvariants} for a detailed discussion.
This completes the diagram in Figure~\ref{fig:square}.
 
Furthermore, as we see in the diagram we get an explicit isomorphism between a coinvariant space of tensors and the corresponding invariant space:
\[
\I_{m,\ell} : \big(\Sym_{\ell m}(\IF^{\ell \times\ell \times 2})\big)_{\SL_\ell(\oIF)\times \SL_\ell(\oIF)} \to \big(\Sym^{\ell m}(\IF^{m \times m \times 2})\big)^{\SL_m(\oIF)\times \SL_m(\oIF)}
\]
via $\I_{m,\ell} := (\KP_{\ell,m})^{-1}\circ\Hermite_{m,\ell}\circ(\KP_{m,\ell}^\star)^{-1}$.
Again, the existence of such an isomorphism was only known in characteristic zero, and only via the character theory of the symmetric group and Schur-Weyl duality, see for example \cite[Lem.~4.4.7]{Ike:12}.

\begin{lemma}\label{lem:duality}
For a group $G$ and every $G$-module $V$ over a field $\IF$, we have an isomorphism of $G$-modules $(\Sym^nV)^\star\to \Sym_n V^\star$ and $(\Wedge^n V)^\star\cong \Wedge^n V^\star$.
\end{lemma}
\begin{proof}
The first isomorphism is proved in \cite[Prop.~3.7]{McD21}, and we generalise it in Proposition~\ref{pro:dual}.
The second isomorphism is proved in \cite[Lem.~3.1]{MW:22}.
\end{proof}

\subsection{Representation theoretic decompositions and \texorpdfstring{$q$}{q}-binomial coefficients}
A \emph{partition} $\la=(\la_1,\ldots,\la_k)$ is a finite list of nonincreasing strictly positive natural numbers.
The number $l(\la)=k$ is called the \emph{length} of $\la$. For $k>l(\la)$, we write $\la_k=0$.
Let $|\la|=\sum_i\la_i$.
We write $\la\vdash_n N$ if $\la$ is a partition with $|\la|=N$ and $l(\la)\leq n$.
We just write $\la\vdash_n$ if $l(\la)\leq n$ with no restriction on~$|\la|$.
We use the notation $(\ell^m)=(\ell,\ell,\ldots,\ell) \vdash \ell m$.
We write $\la\subseteq\mu$ if $\forall i: \la_i\leq \mu_i$.
To every partition $\la$ we associate its Young diagram, which is a top-left justified set of boxes, $\la_i$ boxes in the $i$-th row.
For example, the Young diagram to $\la=(4,2,2,1)$ is
$\ytableausetup{smalltableaux}
\begin{ytableau}
 *(gray) & *(gray) & *(gray) & *(gray) \\
 *(gray) & *(gray) \\
 *(gray) & *(gray) \\
 *(gray)
\end{ytableau}$.
We will always use gray boxes for Young diagrams in this paper, which makes it easier to draw the surrounding grid of potential boxes.
Let $\mathscr P_k(\ell,m) := \{\la \vdash k, \ \la\subseteq (m^\ell)\}$, and let $p_k(\ell,m) := |\mathscr P_k(\ell,m)|$.
Note that $p_k(\ell,m)=p_k(m,\ell)$, which can be seen via transposing the partitions, i.e., reflecting the Young diagram at the main diagonal.
Let $\la^T$ denote the partition to the transposed Young diagram of $\la$, for example $(4,2,2,1)^T=(4,3,1,1)$.

The irreducible polynomial representations $S^\la\IC^n$ of the general linear group $\GL_n(\IC)$ are indexed by partitions $\la\vdash_n$,
 see for example~\cite{FH:13}.
The plethysm coefficient $a_\nu(m[\ell])$ is defined as the multiplicity of $S^\nu\IC^n$ in $\Sym^m \Sym^\ell \IC^n$, which turns out to be independent of $n$, provided $l(\nu)\leq n$.
Finding a sign-free combinatorial interpretation for the plethysm coefficient is Problem 9 in \cite{Sta:00}.

For three partitions $\la\vdash_k d$, $\mu\vdash_\ell d$, $\nu \vdash_n d$, the Kronecker coefficient $\k(\la,\mu,\nu)$ is the multiplicity of $S^\la \IC^k\otimes S^\mu \IC^\ell\otimes S^\nu \IC^n$ in the $\GL_k(\IC)\times\GL_\ell(\IC)\times\GL_n(\IC)$-representation $\Sym^d(\IC^k\otimes\IC^\ell\otimes\IC^n)$.
Finding a sign-free combinatorial interpretation for the Kronecker coefficient is Problem 10 in \cite{Sta:00}.
The special case where $k=\ell$, $d=m\ell$ for some $m$, $\la=\mu=(m^\ell)$, is called the rectangular Kronecker coefficient $\k((m^\ell),(m^\ell),\nu)$.
It equals the multiplicity of $S^\nu\IC^n$ in the $\GL_n(\IC)$-representation $\Sym^d(\IC^\ell\otimes\IC^\ell\otimes\IC^n)^{\SL_\ell\times\SL_\ell}$,
provided $l(\nu)\leq n$.
Although plethysm and Kronecker coefficients seem unrelated at first, several equalities, inequalities, and common constructions are known \cite{Man:11,Ike:12,IP:17,IMW:17,FI:20}.
Both coefficients play an important role in geometric complexity theory, see for example \cite{BLMW11,Bur16,Bur24}.

In this paper, we focus on the case $n=2$.
In this case, both coefficients coincide:
\begin{equation}\label{eq:kab}
\k((m^\ell),(m^\ell),(\ell m-k,k)) = a_{(\ell m -k,k)}(\ell[m]) = b_k(\ell,m),
\end{equation}
where $b_k(\ell,m) = p_k(\ell,m)-p_{k-1}(\ell,m)$.
For the plethysm coefficient, this result can be found in \cite[Cor.~4.2.8]{Stu:08}.
For the Kronecker coefficient, this is proved in \cite{PP:13,PP:14}, and used to prove the strict unimodality of the coefficient sequence of the Gaussian binomial coefficient $\binom{m+\ell}{m}_q = \sum_{n=0}^{\ell m} p_n(\ell,m) q^n$.
The first \emph{combinatorial} proof for $b_k(\ell,m)\geq 0$ was given in \cite{OHa:90}.
Our isomorphism $\KP_{m,\ell}$ is the first explicit isomorphism for \eqref{eq:kab}, even in characteristic zero.

Since $b_k(\ell,m)$ is symmetric in $\ell$ and $m$, we have
$\k((m^\ell),(m^\ell),(\ell m-k,k)) = \k((\ell^m),(\ell^m),(\ell m-k,k))$.
Our isomorphism $\I_{m,\ell}$ gives an explicit isomorphism for this identity.
 
If $\la$ has more than 2 rows, then there are examples for which $\k((m^\ell),(m^\ell),\la)<a_\la(\ell[m])$
and others for which $\k((m^\ell),(m^\ell),\la)>a_\la(\ell[m])$,
for example
$\k((2^2),(2^2),(1^4))=1>0=a_{(1^4)}(2[2])$
and
$\k((3^{12}),(3^{12}),(13^2,2^5))=0<1=a_{(13^2,2^5)}(12[3])$.

\section{Combining symmetric tensors and polynomials}
Elements in the symmetric power $\Sym^m W$ are called polynomials.
In this section we introduce an equivariant product of a polynomial with a symmetric tensor.
For a tensor $t \in \tensor^d W$, we write $[t]_\sym := \sum_{s\in\aS_d t} s$,
where $\aS_d t = \{s \mid \exists \pi: \pi t = s\}$ is the orbit of $t$.
Note that if $t$ has trivial stabilizer under the action of $\aS_d$, then this is the same as $\sum_{\pi\in\aS_d} \pi t$.
There is also an action of $\aS_d$ on $\{1,\ldots,\ell\}^d$ via $\pi(\la_1,\ldots,\la_d)=(\la_{\pi^{-1}(1)},\ldots,\la_{\pi^{-1}(d)})$, and we denote the orbit by $\aS_d \la$. We write $(1^{\mu_1} 2^{\mu_2}\cdots)$ for the list of $d$ numbers that starts with $\mu_1$ many 1s, followed by $\mu_2$ many 2s, and so on.
For three representations $A$, $B$, $C$ of a group $G$, a bilinear map $f:A\times B\to C$ is called $G$-equivariant if $\forall g\in G$, $a\in A$, $b\in B$, we have $f(ga,gb)=gf(a,b)$.

Let $G$ and $H$ be groups, let $V$ be a $G$-representation over a field~$\IF$, and $W$ be an $H$-representation over~$\IF$.
The product group $G\times H$ acts linearly on $\Sym^d V$ via $(g,h)f = gf$,
and $G\times H$ acts linearly on $\Sym_d W$ via $(g,h)t = ht$,
and $G\times H$ acts linearly on $\Sym^d(V \otimes W)$ via $(g,h)((v_{1}\otimes w_{1})\cdots (v_{d}\otimes w_{d})) = (g v_{1})\otimes (h w_{1})\cdots (g v_{d})\otimes (h w_{d})$.
We now define a $G\times H$-equivariant bilinear map
\[
\boxtimes: \Sym^d V \times \Sym_d W \to \Sym^d(V\otimes W).
\]
The construction will be general, and we get our map by instantiating $A = \tensor^d V$, $B = \tensor^d W$, $C = \tensor^d(V\otimes W)$, $S=\aS_d$, $P=G\times H$,
and the bilinear map $\psi: \tensor^d V \times \Sym_d W \to \tensor^d(V\otimes W)$ that
is the restriction of the standard bilinear map $\tensor^d V \times \tensor^d W \to \tensor^d(V\otimes W)$.

In the general situation,
let $S,P$ be groups, let $\mathcal{G} = S\times P$, and let $A$, $B$, $C$ be $\mathcal{G}$-representations. Let $B^S$ denote subspace of $S$-invariants, i.e., the trivial $S$-subrepresentation.
Note that $B^S$ is also a $\mathcal{G}$-representation.
We define the coinvariant space $A_S$ as the quotient
$A_S := A/\langle sw-w\mid w\in A, s\in S\rangle_\IF$.
For $a\in A$ we write $[a] \in A_S$ for the coset of~$a$.
The space of coinvariants $A_S$ is a $\mathcal{G}$-representation on which $S$ acts trivially.
For $p\in P$ we define
\begin{equation}\label{eq:pa}
p[a]=[pa],
\end{equation}
which is well-defined: For $a,a'$ with $[a]=[a']$
we have $[pa] = [pa + pa' -pa'] = [pa'] + [p(a-a')] = [pa']$,
because $[p(a-a')]=[0]$, which can be seen as follows.
$[a]=[a']$ implies $a-a'\in \langle sw-w \mid w \in A,s\in S\rangle$, say $a-a'=\sum_i (s_i w_i - w_i)$,
thus we get $p \sum_i (s_i w_i - w_i) = \sum_i (s_i p w_i - p w_i) \in \langle sw-w \mid w \in A,s\in S\rangle$, therefore $p(a-a')\in \langle sw-w \mid w \in A,s\in S\rangle$, which implies $[p(a-a')]=[0]$.

Let $\psi: A \times B^S \to C$ be a $\mathcal{G}$-equivariant bilinear map,
i.e.,
$\forall g\in \mathcal{G}, a \in A, b \in B^S: \psi(ga,gb)=g(\psi(a,b))$.

Let $b\in B^S$. Then there is an $S$-equivariant linear map $A\to C, \ a \mapsto \psi(a,b)$.
We apply the coinvariant functor to obtain a map $\psi_b : A_S\to C_S$, $[a]\mapsto [\psi(a,b)]$.
Hence, we get a well-defined map
\[
\kappa: B^S \to \Hom(A_S, C_S), \qquad b \mapsto \psi_b.
\]
The map $\kappa$ is linear, because $\kappa(\alpha b+b')([a])=\psi_{\alpha b+b'}([a])=[\psi(a,\alpha b+b')]=[\psi(a,b)]+\alpha [\psi(b')]=(\kappa(b)+\alpha\kappa(b'))[a]$.
This gives the bilinear map
\begin{equation}\label{eq:ASBSCS}
\boxtimes : A_S \times B^S \to C_S, \qquad ([a],b) \mapsto \psi_b(a) = [\psi(a,b)].
\end{equation}
The map $\boxtimes$ is $P$-equivariant, because $p[a]\boxtimes pb \stackrel{\eqref{eq:pa}}{=} [pa]\boxtimes pb =    [\psi(pa,pb)] = [p\psi(a,b)] \stackrel{\eqref{eq:pa}}{=} p [\psi(a,b)]=  p ([a]\boxtimes b)$.

In our special case, we obtain the $G\times H$-equivariant bilinear map
\[
\boxtimes: \Sym^d V \times \Sym_d W \to \Sym^d(V\otimes W),
\]
and \eqref{eq:ASBSCS} directly gives an interpretation on standard basis vectors as
\begin{equation}\label{eq:defboxtimes}
(x_{i_1} \cdots x_{i_d}) \boxtimes [z_1^{\otimes \mu_1} \otimes z_2^{\otimes \mu_2} \otimes \cdots]_{\sym}
 \ = \ 
\sum_{(s_1,\ldots,s_d)\in\aS_d(1^{\mu_1} 2^{\mu_2}\cdots)}
(x_{i_1}\otimes z_{s_1}) \cdots (x_{i_d}\otimes z_{s_d}).
\end{equation}

\section{The ring of invariants}
In this section we determine the invariant ring $(\Sym^{\bullet}(\IF^{\ell\times\ell\times2}))^{\SL_\ell(\oIF)\times \SL_\ell(\oIF)}$.

Let $\{x,y\}$ be a basis of $\IF^2$.
We write $x^ky^{\ell-k}$ for the coset $x^{\otimes k}\otimes y^{\otimes (\ell-k)} + \langle x\otimes y-y\otimes x \rangle$ in $\Sym^\ell\IF^2$.
The basis vectors of $\Sym_\ell \IF^2$ have the form 
\[
F(k) := [x^{\otimes k}\otimes y^{\otimes (\ell-k)}]_\sym
\]
for $0\leq k\leq \ell$.
For $\ell\in\IN$ let
\[
\det_{\ell,\IF}=\sum_{\pi\in\aS_\ell}\sgn(\pi)\prod_{i=1}^\ell \,x_{i,\pi(i)} \ \in \ \Sym^\ell(\IF^{\ell\times \ell})
\]
denote the determinant polynomial.
Clearly $\det_{\ell,\oIF} \in \Sym^{\ell}(\oIF^{\ell\times\ell})^{\SL_\ell(\oIF)\times \SL_\ell(\oIF)}$.
Recall the inclusion $\Sym^\ell(\IF^{\ell\times \ell})\subseteq\Sym^\ell(\oIF^{\ell\times \ell})$.
Since all coefficients of $\det_{\ell,\oIF}$ are in $\IF$, $\det_{\ell,\oIF} \in \Sym^\ell(\IF^{\ell\times \ell})$.
Therefore,
\begin{equation}\label{eq:detinv}
\det_{\ell,\oIF} \in \Sym^{\ell}(\IF^{\ell\times\ell})^{\SL_\ell(\oIF)\times \SL_\ell(\oIF)}.
\end{equation}
We use the short notation $\det_\ell := \det_{\ell,\oIF} = \det_{\ell,\IF}$.

\begin{definition}
For $0\leq k\leq \ell$
we define $M_\ell(k)\in \Sym^{\ell}(\IF^{\ell\times \ell\times 2})$ via
$
M_\ell(k) := \det_\ell \boxtimes F(k).
$
\end{definition}
We remark that $M_\ell(k)$ is the coefficient of $t^k$ in $\det(tX + Y)$, where $X=(x_{i,j})_{1\leq i,j\leq \ell}$ and $Y=(y_{i,j})_{1\leq i,j\leq \ell}$ are variable matrices.

\begin{claim}
$M_\ell(k) \in (\Sym^{\ell}(\IF^{\ell\times\ell\times 2}))^{\SL_\ell(\oIF)\times \SL_\ell(\oIF)}$.
\end{claim}
\begin{proof}
Let $G = \SL_\ell(\oIF)\times \SL_\ell(\oIF)$.
By \eqref{eq:detinv} we have that $\det_\ell$ is $G$-invariant.
Let $H = \{1\}$ be the trivial subgroup of $\GL_2(\oIF)$.
By the $G\times H$-equivariance of $\boxtimes$, for $(g,1)\in G\times H$ we have
$(g,1) M_\ell(k) = (g \det_\ell) \boxtimes F(k) = \det_\ell \boxtimes F(k) = M_\ell(k)$,
because $\det_\ell$ is $G$-invariant.
\end{proof}

The next theorem is known over $\IC$, for example via quiver representations, see \cite{derksen2017introduction},\cite{SW2000}.

\begin{theorem} \label{thm:basis}
The algebra
$(\Sym^{\bullet}(\IF^{\ell\times\ell\times2}))^{\SL_\ell(\oIF)\times \SL_\ell(\oIF)}$
is generated by the set
$\left\{ M_\ell(k) \mid 0\leq k \leq \ell\right\}$.
\end{theorem}
The rest of this section is dedicated to proving this theorem.
We interpret $\IF^{\ell\times\ell\times 2}=\IF^{\ell\times\ell}\oplus\IF^{\ell\times\ell}$ as a space of pairs of matrices.
Let $\mathbb{I}_\ell$ denote the $\ell\times\ell$ identity matrix, and let $\Diag(\mu_1,\ldots,\mu_\ell)$ denote the $\ell\times\ell$ diagonal matrix with $\mu_i$ on the main diagonal.
We start by defining a homomorphism of graded $\IF$-algebras:
\begin{align*}
\Phi : (\Sym^{\bullet}(\IF^{\ell\times\ell\times2}))^{\SL_\ell(\oIF)\times \SL_\ell(\oIF)}
&\to \IF[\mu_1,\dots,\mu_\ell,\nu],
\\
p &\mapsto p(\nu \mathbb{I}_\ell,\Diag(\mu_1,\ldots,\mu_\ell)).
\end{align*}
Let $e_k$ denote the $k$-th elementary symmetric polynomial in the variables $\mu_1,\ldots,\mu_\ell$.
By definition, we have
\begin{equation}\label{eq:PhiMe}
\Phi(M_\ell(k)) = \nu^{k} e_{\ell-k}.
\end{equation}
We see that
\begin{equation}\label{eq:Mkalgindep}
\textup{$\{M_\ell(k)\mid 0\leq k\leq \ell\}$ are algebraically independent},
\end{equation}
because if
there exists a polynomial $P$ with $P(e_0 \nu^{\ell},\ldots,e_\ell \nu^{0})=0$, then there exists $P'$ with $P'(\nu,e_1,\ldots,e_\ell)=0$, but the $e_k$ are algebraically independent (\cite[Thm.~8.2]{Lan05}), and $\nu$ is a variable unused by the $e_k$.

\begin{claim}
\label{cla:inj}
$\Phi$ is injective.
\end{claim}
\begin{proof}
The map $\Phi$ depends on the field $\IF$, so we write $\Phi_\IF$ to be precise.
Note that for every $p\in (\Sym^{\bullet}(\IF^{\ell\times\ell\times 2}))^{\SL_\ell(\oIF)\times \SL_\ell(\oIF)}
\subseteq
(\Sym^{\bullet}(\oIF^{\ell\times\ell\times 2}))^{\SL_\ell(\oIF)\times \SL_\ell(\oIF)}$
we have $\Phi_{\IF}(p)=\Phi_{\oIF}(p)$.
This implies $\ker(\Phi_\IF)\subseteq \ker(\Phi_\oIF)$.
Let $p \in \ker(\Phi_\IF) \subseteq \ker(\Phi_\oIF)$.
We interpret $p$ as a polynomial on $\oIF^{\ell\times\ell\times 2} = \oIF^{\ell\times\ell}\oplus\oIF^{\ell\times\ell}$.
We study the evaluation $p(X,Y)$, $(X,Y) \in \oIF^{\ell\times\ell}\oplus\oIF^{\ell\times\ell}$.
As a first step, consider the case where $X$ has full rank,
and let $\nu\in\oIF$ such that $\nu^\ell=\det(X)$. Note that $\nu \, X^{-1} \in \SL_\ell(\oIF)$.
Consider the subcase where $\nu \, X^{-1} Y$ has distinct eigenvalues in $\oIF$ (it follows that $\nu \, X^{-1} Y$ is diagonalizable over $\oIF$),
and let $A^{-1} \nu \, X^{-1} Y A = \Diag(\mu_1,\ldots,\mu_\ell)$ for $A \in \SL_\ell(\oIF)$.
Let $\mathbb{I} = \mathbb{I}_\ell$.
Using the invariance of $p$, we see that
\begin{eqnarray*}
p(X,Y) &=&
p(\nu\, X^{-1} X,\nu\, X^{-1} Y)
=
p(\nu\, \mathbb{I}, \nu\, X^{-1}Y)
\\
&=&
p(A^{-1}\nu\, \mathbb{I}\, A, A^{-1}\nu\, X^{-1} YA)
=
p(\nu\, \mathbb{I}, \Diag(\mu_1,\ldots,\mu_\ell)) \stackrel{p \in \ker(\Phi_{\oIF})}{=} 0.
\end{eqnarray*}
The full rank condition of $X$ is a Zariski-open condition on $\oIF^{\ell\times\ell\times 2}$.
The matrix $\nu X^{-1} Y$ has distinct eigenvalues if and only if $\adj(X) Y$ has distinct eigenvalues, where $\adj(X)$ is the adjugate matrix of $X$, i.e., the transpose of the cofactor matrix.
But $\adj(X) Y$ has distinct eigenvalues if and only if the characteristic polynomial of $\adj(X) Y$ does not have a repeated root,
which happens if and only if the discriminant (the determinant of the Sylvester matrix of the polynomial and its derivative) of the characteristic polynomial of $\adj(X) Y$ does not vanish.
This is also a Zariski-open condition on $\oIF^{\ell\times\ell\times 2}$.
Hence, $p(X,Y)=0$ for a Zariski-dense subset of $\oIF^{\ell\times\ell\times 2}$, and thus $p=0$ is the zero polynomial.
\end{proof}

Since $\Phi$ is injective, we have the isomorphism $(\Sym^{\bullet}(\IF^{\ell\times\ell\times 2}))^{\SL_\ell(\oIF)\times \SL_\ell(\oIF)}\cong \Im(\Phi)$. The next goal is to determine $\Im(\Phi)$, see Claim~\ref{cla:imageelemsym} below. We start with some observations.
\begin{claim}\label{cla:multipleofell}
For every $p \in (\Sym^{\bullet}(\IF^{\ell\times\ell\times 2}))^{\SL_\ell(\oIF)\times \SL_\ell(\oIF)}$, we have that $d:=\deg(p)$ is a multiple of $\ell$.
\end{claim}
\begin{proof}
Define $q(\kappa_1,\dots,\kappa_\ell,\mu_1,\dots,\mu_\ell) := p\big(\Diag(\kappa_1,\dots,\kappa_\ell),\Diag(\mu_1,\dots,\mu_\ell)\big)$.
For any $\alpha\in\oIF$, let $A_\alpha := \Diag(\alpha,\alpha^{-1},1,\dots,1) \in \SL_\ell(\oIF)$.
Due to the invariance of $p$, we have
\begin{align*}
q(\kappa_1,\dots,\kappa_\ell,\mu_1,\dots,\mu_\ell) &= 
p\big(A_\alpha\Diag(\kappa_1,\dots,\kappa_\ell),A_\alpha\Diag(\mu_1,\dots,\mu_\ell)\big)
\\
&=
p\big(\Diag(\alpha\kappa_1,\alpha^{-1}\kappa_2,\kappa_3,\dots,\kappa_\ell),\Diag(\alpha\mu_1,\alpha^{-1}\mu_2,\mu_3,\dots,\mu_\ell)\big)
\\
&=
q(\alpha\kappa_1,\alpha^{-1}\kappa_2,\kappa_3,\dots,\kappa_\ell,\alpha \mu_1,\alpha^{-1}\mu_2,\mu_3,\dots,\mu_\ell).
\end{align*}
By the equation above, since $\oIF$ is infinite, in every monomial $m$ of $q$ we have $\deg_{\kappa_1}(m)+\deg_{\mu_1}(m)=\deg_{\kappa_2}(m)+\deg_{\mu_2}(m)$.
Analogously, for all $j$, we have $\deg_{\kappa_1}(m)+\deg_{\mu_1}(m)=\deg_{\kappa_j}(m)+\deg_{\mu_j}(m)$.
Hence,
\[
d = {\textstyle\sum_{j=1}^\ell}
\big(\deg_{\kappa_j}(m)
+
\deg_{\mu_j}(m)\big)
= \ell (
\deg_{\kappa_1}(m)
+
\deg_{\mu_1}(m)
).
\qedhere
\]
\end{proof}
\begin{claim}\label{cla:symmetricinmu}
$\Phi(p)$ is symmetric in the variables $\mu_1,\ldots,\mu_\ell$.
\end{claim}
\begin{proof}
For each transposition $\sigma = (i \, j) \in \aS_\ell$ we have that the product $Q_{i,j}:=\Diag(-1,1,1,\ldots,1)\cdot P_{(i \, j)} \in \SL_\ell(\oIF)$, where $P_{(i \, j)}$ is the permutation matrix of $(i \, j)$. Hence, due to the invariance of $p$ we have
\begin{align*}
\Phi(p)(\nu,\mu_1,\dots,\mu_\ell)
&=
p\big(\nu\, \mathbb{I},\mathsf{Diag}(\mu_1,\dots,\mu_\ell)\big)= p\big(\nu\cdot Q_{i,j}\, \mathbb{I}\, Q_{i,j}^{-1},Q_{i,j}\Diag(\mu_1,\dots,\mu_\ell)Q_{i,j}^{-1}\big)
\\
&=p(\nu\, \mathbb{I},\mathsf{Diag}(\mu_{\sigma(1)},\dots,\mu_{\sigma(\ell)}))= \Phi(p)(\nu,\mu_{\sigma(1)},\dots,\mu_{\sigma(\ell)}).\qedhere
\end{align*}
\end{proof}
Combining Claim~\ref{cla:multipleofell} and Claim~\ref{cla:symmetricinmu}, it follows that
\begin{equation}\label{eq:degmui}
\textstyle\deg_{\mu_i}(p)\leq\frac{d}{\ell}.
\end{equation}
We are now ready to determine the image of $\Phi$.
\begin{claim}
\label{cla:imageelemsym}
$\Im(\Phi) = \IF\left[\nu^k e_{\ell-k}\mid 0\leq k \leq \ell\right]$.
\end{claim}
\begin{proof}
We have
$\Phi(M_\ell(k)) = \nu^k e_{\ell-k}$, see \eqref{eq:PhiMe}, so it suffices to show that
$\Im(\Phi) \subseteq \IF\left[\nu^k e_{\ell-k}\mid 0\leq k \leq \ell\right]$.
Since the group action of $\SL_\ell(\oIF)\times \SL_\ell(\oIF)$ preserves degrees, every invariant decomposes into a sum of homogeneous invariants.
Let $p \in (\Sym^{\bullet}(\IF^{\ell\times\ell\times 2}))^{\SL_\ell(\oIF)\times \SL_\ell(\oIF)}$
be homogeneous of some degree $d$.
It remains to show that $\Phi(p) \in \IF\left[\nu^k e_{\ell-k}(\mu_1,\dots,\mu_\ell)\mid 0\leq k \leq \ell\right]$.
We collect powers of $\nu$ in $\Phi(p)$ and express the symmetric part (Claim~\ref{cla:symmetricinmu}) as a polynomial in elementary symmetric polynomials:
\begin{equation}
\label{eq:pprimeine}
\Phi(p)(\nu,\mu_1,\dots,\mu_\ell)
 \ = \ 
\sum_{\delta=0}^{d} \nu^\delta \sum_{\mathbf{d}=(d_1,\ldots,d_\ell)} \alpha_{\mathbf{d}} \, e_1^{d_1} \cdots e_\ell^{d_\ell}
\end{equation}
for constants $\alpha_{\mathbf{d}}\in \IF$,
and the sum for every $\delta$ is over
$(d_1,\ldots,d_\ell) \in \IN_0^\ell$ with
\begin{equation}
\label{eq:1}
\textstyle
\sum_{k=1}^\ell d_k \cdot k = d-\delta,
\end{equation}
since $\deg (e_k)=k$.
Combining \eqref{eq:degmui} with the fact that $\deg_{\mu_i}(e_k) = 1$, we also obtain
\begin{equation}
\label{eq:2}
\textstyle
\sum_{k=1}^\ell d_k \leq \frac{d}{\ell}.
\end{equation}
Therefore, for every $\delta$ we have
\[\textstyle
\delta
\stackrel{\eqref{eq:1}}{=}
d - \sum_{k=1}^\ell d_k \cdot k
\stackrel{\eqref{eq:2}}{\geq}
\sum_{k=1}^\ell d_k (\ell-k),
\]
hence $\delta' := \delta-\sum_{k=1}^\ell d_k (\ell-k) \geq 0$.
Therefore we can rewrite the monomial $\nu^\delta e_1^{d_1}e_2^{d_2}\cdots e_\ell^{d_\ell}$
from \eqref{eq:pprimeine}
as
\begin{equation}
\label{eq:newmonomial}
\nu^{\delta'} (\nu^{\ell-1}e_1)^{d_1}(\nu^{\ell-2}e_2)^{d_2}\cdots (\nu e_{\ell-1})^{d_{\ell-1}} (\nu^0 e_\ell)^{d_\ell}.
\end{equation}
Since $d$ is a multiple of $\ell$ (Claim~\ref{cla:multipleofell}),
we can also see from \eqref{eq:1} that $\delta'= d - \ell\cdot \sum_{k=1}^\ell d_k$, so $\delta'$ is a multiple of $\ell$,
hence $\nu^{\delta'} = (\nu^\ell e_0)^{d_0}$ with $d_0 = \frac{d}{\ell}-\sum_{k=1}^\ell d_k$.
Overall, the monomial in \eqref{eq:newmonomial} is contained in $\IF\left[\nu^k e_{\ell-k}(\mu_1,\dots,\mu_\ell)\mid 0\leq k \leq \ell\right]$, which implies that $\Phi(p)$ is also contained there, as desired.
\end{proof}
\begin{proof}[Proof of Theorem~\ref{thm:basis}]
Combining Claim~\ref{cla:inj} and Claim~\ref{cla:imageelemsym}, we have the isomorphism
\[
(\Sym^{\bullet}(\IF^{\ell\times\ell\times2}))^{\SL_\ell(\oIF)\times \SL_\ell(\oIF)} \ \cong \  \IF\left[\nu^k e_{\ell-k}(\mu_1,\dots,\mu_\ell)\mid 0\leq k \leq \ell\right],
\]
and since the generators of the right-hand side ring are the images of the polynomials $M_\ell(k)$, we conclude that
\[
(\Sym^{\bullet}(\IF^{\ell\times\ell\times2}))^{\SL_\ell(\oIF)\times \SL_\ell(\oIF)}
 \ = \ 
\IF\left[M_\ell(k)\mid 0\leq k \leq \ell\right].
\qedhere
\]
\end{proof}

Taking the homogeneous degree $m$ component in Theorem~\ref{thm:basis}, we get the following immediate corollary.
\begin{corollary}
\label{cor:basiscount}
The set 
$
\left\{ \prod_{i=1}^m M_\ell(\lambda_i)\mid \lambda \subseteq (\ell^m) \right\}
$
is a basis of the
$\IF$-vector space $(\Sym^{\ell m}(\IF^{\ell\times\ell\times 2}))^{\SL_\ell(\oIF)\times \SL_\ell(\oIF)}$.
In particular, $\dim\big((\Sym^{\ell m}(\IF^{\ell\times\ell\times 2}))^{\SL_\ell(\oIF)\times \SL_\ell(\oIF)}\big) = \binom{m+\ell}{\ell}$.
\end{corollary}

\section{The Kronecker-Plethysm isomorphism}
We now change the grading of the algebra
$(\Sym^{\bullet}(\IF^{\ell\times\ell\times 2}))^{\SL_\ell(\oIF)\times \SL_\ell(\oIF)}$.
We define the re-graded algebra
$(\Sym^{\ell\bullet}(\IF^{\ell\times\ell\times 2}))^{\SL_\ell(\oIF)\times \SL_\ell(\oIF)}[1/\ell]$ via defining its homogeneous degree $m$ components as
\[
\left((\Sym^{\ell\bullet}(\IF^{\ell\times\ell\times 2}))^{\SL_\ell(\oIF)\times \SL_\ell(\oIF)}[1/\ell]\right)_m = (\Sym^{\ell m}(\IF^{\ell\times\ell\times 2}))^{\SL_\ell(\oIF)\times \SL_\ell(\oIF)}.
\]
Note that $M_\ell(k)$ is homogeneous of degree 1 in this algebra.

\begin{definition}
Let $\ell\in \IN$.
We define the homomorphism of graded $\IF$-algebras
\[
\KP_{\ell} :(\Sym^{\ell\bullet}(\IF^{\ell\times\ell\times 2}))^{\SL_\ell(\oIF)\times \SL_\ell(\oIF)}[1/\ell] \ \longrightarrow \ \Sym^\bullet \Sym_\ell(\IF^2)
\]
via defining it on generators (see Theorem~\ref{thm:basis}) as
$\KP_{\ell}(M_\ell(k)) := F(k)$.
\end{definition}
\eqref{eq:Mkalgindep} implies that this is well-defined.

\begin{proposition}\label{pro:Kliso}
$\KP_{\ell}$ is an isomorphism of graded $\IF$-algebras. \end{proposition}
\begin{proof}
Since $\{F(k)\mid 0\leq k \leq \ell\}$ is a basis of the $\IF$-vector space $\Sym_\ell(\IF^2)$, the $F(k)$ generate the algebra $\Sym^\bullet \Sym_\ell(\IF^2)$ and are algebraically independent therein \cite[Pro.~8.1]{Lan12}.
The inverse $\KP_{\ell}^{-1}$ of $\KP_{\ell}$ is defined on these generators as $\KP_{\ell}^{-1}(F(k)) := M_\ell(k)$.
\end{proof}

\begin{theorem}\label{thm:kpiso}
Let $\ell,m\in \IN$.
Let the isomorphism of $\IF$-vector spaces
\[
\KP_{m,\ell} : (\Sym^{\ell m}(\IF^{\ell\times\ell\times 2}))^{\SL_\ell(\oIF)\times \SL_\ell(\oIF)} \ \stackrel{\sim}{\longrightarrow} \ \Sym^m\Sym_\ell(\IF^2)
\]
be defined as the restriction of the isomorphism $\KP_{\ell}$ to 
the degree $m$ homogeneous components.
We have that $\KP_{m,\ell}$ is an isomorphism of $\GL_2(\IF)$-representations.
\end{theorem}

\begin{proof}
It suffices to show that $\KP_{m,\ell}$ respects the group action of $\GL_2(\IF)$.
For $\la \subseteq (\ell^m)$ we write
$F(\lambda) := \prod_{i=1}^m F(\lambda_i)$ and
$M(\lambda) := \prod_{i=1}^m M_\ell (\lambda_i)$.
Since $\{F(k)\mid 0\leq k\leq \ell\}$ is a basis of the $\IF$-vector space $\Sym_\ell(\IF^2)$, we have that $\{F(\la)\mid \la\subseteq (\ell^m)\}$ is a basis of $\Sym^m\Sym_\ell(\IF^2)$.
By Corollary~\ref{cor:basiscount},
$\{M(\la)\mid \la\subseteq (\ell^m)\}$ is a basis of $(\Sym^{\ell m}(\IF^{\ell\times\ell\times 2}))^{\SL_\ell(\oIF)\times \SL_\ell(\oIF)}$.
Let $g\in \GL_2(\IF)$. By the equivariance of $\boxtimes$ we have
\begin{equation}\label{eq:JMlk}
g M_\ell(k) = g(\det_\ell \boxtimes F(k)) = \det_\ell \boxtimes gF(k).
\end{equation}
Let $c_{k,j} \in \IF$ with $g\cdot F(k) = \sum_{j} c_{k,j} F(j)$.
\begin{align*}
(\KP_{\ell}\circ g) (M_\ell(\la)) &= (\KP_{\ell}\circ g) (\prod_i M_\ell(\la_i)) = \prod_i (\KP_{\ell}\circ g)(M_\ell(\la_i))
\stackrel{\eqref{eq:JMlk}}{=} \prod_i \KP_{\ell}(\det_\ell \boxtimes (g\cdot F(\la_i))
\\
&=
\prod_i \KP_{\ell}\left(\det_\ell \boxtimes\sum_{j} c_{\la_i,j} F(j)\right) =
\prod_i \sum_j c_{\la_i,j} \, \KP_{\ell}(M_\ell(j)) = \prod_i \sum_{j}c_{\la_i,j} F(j)
\\
&=
\prod_i g\cdot F(\la_i)
=
g \cdot F(\lambda) = (g \circ \KP_{\ell})(M_\ell(\lambda)).
\qedhere
\end{align*}
\end{proof}

\begin{corollary}
We also have the following isomorphisms of $\GL_2(\IF)$-representations:
\begin{itemize}
    \item $
    (\Sym_{\ell m}(\IF^{\ell\times\ell\times 2}))_{\SL_\ell(\oIF)\times \SL_\ell(\oIF)}\cong \Sym_m\Sym^\ell(\IF^2) 
    $
    \item$
    (\Sym_{\ell m}(\IF^{\ell\times \ell\times 2}))_{\SL_\ell(\oIF)\times \SL_\ell(\oIF)}\cong (\Sym^{\ell m}(\IF^{m\times m\times 2}))^{\SL_m(\oIF)\times \SL_m(\oIF)}
    $
\end{itemize}
\end{corollary}

\begin{proof}
The first isomorphism is obtained by taking the dual of the map $\KP_{m,\ell}$. The spaces in this isomorphism are the duals of the spaces in Theorem~\ref{thm:kpiso}, as it is shown in Lemma~\ref{lem:duality} and Section~\ref{sec:coinvariants} (equation~\eqref{eq:coinvariants}).
The second isomorphism is obtained as the composition of isomorphisms $ (\KP_{\ell,m})^{-1}\circ\Hermite_{m,\ell}\circ(\KP_{m,\ell}^\star)^{-1}$, as it is shown in Figure~\ref{fig:square}.
\end{proof}

\section{Coinvariants}\label{sec:coinvariants}

In this section, we prove that
\begin{equation}\label{eq:coinvariantsdual}
\left(\big(\Sym^{\ell m}\:(\IF^{\ell \times \ell \times 2})\big)^{\SL_\ell(\oIF)\times \SL_\ell(\oIF)}\right)^\star_\IF
\ \cong\ \left(\Sym_{\ell m}\:(\IF^{\ell \times \ell \times 2})\right)_{\SL_\ell(\oIF)\times \SL_\ell(\oIF)}
\end{equation}
as $\SL_2(\IF)$-representations.
The rest of this section is devoted to this proof.

\subsection{Scalar extensions and over-actions}\label{subsec:overact}
In the whole Section~\ref{sec:coinvariants}, let $\IF\subseteq\oIF$ be any field extension, let $V$ be a finite dimensional $\IF$-vector space, and let $W=V\otimes_\IF\oIF$ be its scalar extension.
We will give a concrete $V$ in \S\ref{subsec:specificcase}.
There is an embedding $V\to W$ via $v\mapsto v\otimes 1$. We use this specific embedding to write
\begin{equation}\label{eq:scalextembed}
V\subseteq W.
\end{equation}
Let $V$ carry an $\IF$-linear action of the group~$H$.
Then $H$ also acts $\oIF$-linearly on $W$ via the composition of group homomorphisms
$H \to \GL(V) \to \GL(W)$,
where the last group homomorphism is due to $\GL(V)\subseteq\GL(W)$, which is induced from \eqref{eq:scalextembed}.
Let $Z$ act $\oIF$-linearly on $W$, and let both actions commute on~$W$.
Then we say that ``$Z\times H$ over-acts on $V$'', i.e., $H$ acts $\IF$-linearly on $V$ and $Z$ acts $\oIF$-linearly on $V\otimes_\IF\oIF$.
This is the situation we assume in Section~\ref{sec:coinvariants}.

\subsection{Dual spaces}
For the sake of clarity, for an $\IF$-vector space $V$, we write $V^\star_\IF$ for its dual vector space,
and for an $\oIF$-vector space $W$, we write $W^\star_\oIF$ for its dual vector space.
 
Every $\IF$-basis $B$ of $V$ can be interpreted as an $\oIF$-basis of $W$ via \eqref{eq:scalextembed}.

\begin{claim}\label{cla:dualscalarextension}
There is an isomorphism of $\oIF$-vector spaces
$W^\star_{\oIF} \ \cong \ V^\star_\IF \otimes_\IF \oIF$.
In particular, if $\{b_i\}$ is an $\IF$-basis of $V$, then an isomorphism is given by $(b_i\otimes 1)^\star\mapsto (b_i^\star)\otimes 1$.
\end{claim}
\begin{proof}
Let $\{b_i\}$ be an $\IF$-basis of $V$, hence $\{b_i\otimes 1\}$ is an $\oIF$-basis of $W$.
Let $\{b_i^\star\}$ denote the dual basis of $V$, and let $\{(b_i\otimes 1)^\star\}$ denote the dual basis of $W$.
Then $W^\star_\oIF = \{\sum_{i}\alpha_i(b_i\otimes 1)^\star\mid \alpha_i\in\oIF\}$,
and $V^\star_\IF \otimes_\IF \oIF = \{\sum_{i}\beta_i (b_i^\star \otimes \gamma_i) \mid \beta_i\in\IF, \gamma_i\in\oIF\}$.
The map $\sum_{i}\alpha_i(b_i\otimes 1)^\star\mapsto\sum_{i}b_i^\star \otimes \alpha_i$
is an $\oIF$-isomorphism, with inverse given by
$\sum_{i}\beta_i (b_i^\star \otimes \gamma_i)\mapsto\sum_{i}(\beta_i\gamma_i)(b_i\otimes 1)^\star$.
\end{proof}
Using Claim~\ref{cla:dualscalarextension} and \eqref{eq:scalextembed}, we can write
\begin{equation}\label{eq:VstarWstar}
V^\star_\IF\subseteq W^\star_\oIF.
\end{equation}
For a dual vector $f\in V^\star_\IF$ of $v\in V$, we write $f=v^\star$ without specifying the field,
because $v^\star_\IF$ and $(v\otimes 1)^\star_\oIF$ are mapped to each other under the isomorphism in Claim~\ref{cla:dualscalarextension}.

\begin{claim}\label{cla:semiactdual}
$Z \times H$ over-acts on $V^\star_\IF$.
\end{claim}
\begin{proof}
It suffices to check the group action of $H$ on $V^\star_{\IF}$ (this is standard) and the action of $Z\times H$ on $V^\star_\IF\otimes_\IF\oIF$,
which is standard if one applies the isomorphism with $W^\star_{\oIF}$ from Claim~\ref{cla:dualscalarextension}.
\end{proof}

\subsection{Invariants and coinvariants}

We say that $V$ is an $\IF H$-representation if $V$ is a finite dimensional $\IF$-vector space with an $\IF$-linear group action of~$H$.
Let
\[
I := \langle zw-w\mid z\in Z, w\in W\rangle_\oIF \quad \subseteq W
\]

Vectors in a quotient $W/I$ are denoted by $[w]_I$.
The next lemma is obvious, but we list it for reference.
\begin{lemma}
\label{lem:quotienteq}
For $w,w'\in W$ we have $[w]_I=[w']_I$ if and only if $w-w'\in I$.
\end{lemma}
\begin{proof}
$[w]_I=[w']_I$ \ iff \ $w+I = w'+I$ \ iff \ $(w-w')+I = I$ \ iff \ $w-w' \in I$.
\end{proof}

\begin{claim}\label{cla:IFH}
$I$ is an $\oIF H$-representation, and hence also an $\IF H$-representation.
\end{claim}
\begin{proof}
Let $v = \sum_i \alpha_i (z_i w_i - w_i)$, $\alpha_i\in\oIF$, $z_i\in Z$, $w_i\in W$.
Then
\[\textstyle
\forall h\in H : \ hv = \sum_i \alpha_i (h z_i w_i - h w_i) = \sum_i \alpha_i \big(z_i (h w_i) - (h w_i)\big) \in I. \qedhere
\]
\end{proof}
Define the annihilator
\[
\ann(I)=\{\varphi\in W^\star_\oIF\mid \forall v \in I: \varphi(v)=0\}.
\]

\begin{definition}\label{def:VinvZ}
We define $V^Z := W^Z \cap V$.
\end{definition}
\begin{claim}\label{cla:invH}
$V^Z$ is an $\IF H$-representation.
\end{claim}
\begin{proof}
Let $v \in V^Z$. Then $\forall z \in Z: zhv = hzv = hv$, hence $hv \in V^Z$.
\end{proof}

We combine Claim~\ref{cla:semiactdual} and Claim~\ref{cla:invH} (invoked with $V^\star_\IF$ instead of $V$) to obtain the $\IF H$-representation $(V^\star_\IF)^Z$.

Define the generalization of a coinvariant space:
\[
\coinv := \big\{[v]_I \in W_Z \ \big| \ v \in V\big\}.
\]
\begin{remark}
We remark that $\coinv = \big\{x \in W_Z \ \big| \ x\cap V \neq\emptyset \big\} =: S$, which can be seen as follows.
 
Let $x\in\coinv$, $x=[v]_I$ for some $v \in V$. Hence, $[v]_I\cap V \neq \emptyset$ and thus $[v]_I\in S$.
For the other direction, let 
$x\in S$, $x=[v]_I$ for some $v\in W$.
Then $\exists u \in [v]_I\cap V$, i.e., $u=v+w$ for some $w\in I$. Hence, $u-v\in I$, and by Lemma~\ref{lem:quotienteq} we have $[u]_I=[v]_I$.
Therefore, $[v]_I\in \coinv$.
\end{remark}

\begin{claim}\label{cla:VZFH}
$\coinv$ is an $\IF H$-representation.
\end{claim}
\begin{proof}
Let $[v]_I \in \coinv$, $v\in V$, and let $[w]_I \in \coinv$, $w\in V$, and let $\alpha \in \IF$.
We have $\alpha v + w\in V$ and hence $\alpha[v]_I+[w]_I = [\alpha v+w]_I \in \coinv$.
It remains to verify that $\coinv$ is closed under the action of $H$.
Let $[v]_I \in \coinv$, $v\in V$, $h\in H$.
Then $hv\in V$, and thus $h[v]_I=[hv]_I \in \coinv$.
\end{proof}
Claim~\ref{cla:VZFH} implies that
\begin{equation}\label{eq:VZFH}
\textup{$(\coinv)^\star_\IF$ is also an $\IF H$-representation.}
\end{equation}
\begin{claim}\label{cla:annIFH}
$\ann(I)$ is an $\IF H$-representation.
\end{claim}
\begin{proof}
$\ann(I)$ is an $\oIF$-vector space, hence also an $\IF$-vector space.
Let $f\in\ann(I)$, and let $w\in I$. Then
\[
(hf)(w)=f(\underbrace{h^{-1}w}_{\in I, \  \textup{Cla.}\,\ref{cla:IFH}}) = 0.
\]
Therefore, $hf\in\ann(I)$.
\end{proof}
Claim~\ref{cla:annIFH} implies that $\ann(I)\cap V^\star_\IF$ is also an $\IF H$-representation.
We compare it to \eqref{eq:VZFH} in the following lemma.
\begin{lemma}\label{lem:VZannihilator}
$(\coinv)^\star_\IF \ \cong \ \ann(I) \cap V^\star_\IF$ as $\IF H$-representations.
\end{lemma}
\begin{proof}
Let $p:W\to W/I$ be the canonical map. If $v\in V$, then $p(v) \in \coinv$ by definition. In other words, restricting the domain and codomain of $p$, we obtain $q:V\to \coinv$.
Given $f\in (\coinv)^\star_\IF$, we have $f\circ q \in \ann(I)\cap V^\star_\IF$,
because for $w\in I$ we have $(f\circ q)(w) = f([0]_I)=0$, and for $v\in V$ we have $(f\circ q)(v)=f([v]_I)\in\IF$.
Hence, the composition with $q$ gives an $\IF$-linear map
\[
\Phi:(\coinv)^\star_\IF\to\ann(I)\cap V^\star_\IF.
\]
For every $F\in \ann(I)\cap V^\star_\IF$ we define $f\in(\coinv)^\star_\IF$ via $f([v]_I) := F(v)$, for $v\in V$.
This is well-defined, because firstly, if $[v]_I=[w]_I$, $v,w\in V$, then $v-w\in I$ by Lemma~\ref{lem:quotienteq} and hence
$F(v-w)=0$, which implies $F(v) = F(w)$,
and secondly, $f$ maps to $\IF$, because $F$ maps to $\IF$.
This gives an $\IF$-linear map
\[
\Psi:\ann(I)\cap V^\star_\IF\to(\coinv)^\star_\IF.
\]
\begin{enumerate}
\item $\Psi\circ\Phi = \textup{id}_{\ann(I)\cap V^\star_\IF}$, because $(\Psi\circ\Phi)(f) = \Psi(f\circ q) = f$.
\item $\Phi\circ\Psi = \textup{id}_{(\coinv)^\star_\IF}$, because $(\Phi\circ\Psi)(F) = \Phi(f) = F$.
\end{enumerate}
Hence, $(\coinv)^\star_\IF$ and $\ann(I)\cap V^\star_\IF$ are isomorphic $\IF$-vector spaces.
 
It remains to prove that $\Phi:(\coinv)^\star_\IF\to\ann(I)\cap V^\star_\IF$, $\Phi(f)=f\circ q$,
is $H$-equivariant, i.e., $h(\Phi(f)) = \Phi(hf)$.
Indeed,
\begin{align*}
(h(\Phi(f)))(v)&\overset{\hphantom{q \textup{ equiv}}}{=} (h(f\circ q))(v)
=(f\circ q)(h^{-1}v)
=f(q(h^{-1}v))
\\
&\overset{q \textup{ equiv}}{=}
f(h^{-1} \, q(v))
=(hf)(q(v))
= ((hf)\circ q)(v) = (\Phi(hf))(v).\qedhere
\end{align*}
\end{proof}

\begin{proposition}\label{pro:dual}
We have
$(\coinv)^\star_\IF \ \cong \ (V^\star_\IF)^Z$ as $\IF H$-representations.
\end{proposition}
\begin{proof}
We use Lemma~\ref{lem:VZannihilator} to see that $(\coinv)^\star_\IF \ \cong \ \ann(I) \cap V^\star_\IF$.
It remains to show that $\ann(I)\cap V^\star_\IF = (V^\star_\IF)^Z$.
\begin{eqnarray*}
&&
\varphi \in \ann(I)\cap V^\star_\IF
\\
&\Longleftrightarrow&
\varphi\in V^\star_\IF \textup{ and } \forall u\in I: \ \varphi(u)=0
\\
&\Longleftrightarrow&
\varphi\in V^\star_\IF \textup{ and } \forall z\in Z, w\in W: \ \varphi(zw-w)=0
\\
&\stackrel{\varphi\textup{ linear}}{\Longleftrightarrow}&
\varphi\in V^\star_\IF \textup{ and } \forall z\in Z, w\in W: \ \varphi(zw)=\varphi(w)
\\
&\Longleftrightarrow&
\varphi\in V^\star_\IF
\textup{ and } \forall z\in Z: z^{-1}\varphi=\varphi
\\
&\Longleftrightarrow&
\varphi \in (W^\star_\oIF)^Z \cap V^\star_\IF
 \ \stackrel{\textup{Def.~}\ref{def:VinvZ}}{=} \ 
(V^\star_\IF)^Z.
\end{eqnarray*}

\vspace{-0.8cm}
\end{proof}

Note that using $V^\star_\IF$ instead of $V$ in Proposition~\ref{pro:dual} and taking the dual gives
\begin{equation}\label{eq:Vstardual}
(V^\star_\IF)_Z \ \cong \ (V^Z)^\star_\IF
\end{equation}

\begin{remark}
From the definitions, 
$\Sym_m (W^\star_\IF) =(\tensor^m (W^\star_\IF))^{\aS_m}$
and $(\Sym^m W)^\star_\IF = ((\tensor^m W)_{\aS_m})^\star_\IF$, 
and Proposition~\ref{pro:dual} implies that these are isomorphic (compare Proposition~\ref{lem:duality}).
This was proved directly (i.e., using less general principles) in \cite[Prop.~3.7]{McD21}.
\end{remark}

\subsection{The specific case of interest}
\label{subsec:specificcase}

Let $Z = \SL_\ell(\oIF)\times\SL_\ell(\oIF)$.
 
We have the scalar extension $\Sym^{\ell m} \:\oIF^{\ell \times\ell \times 2} = (\Sym^{\ell m} \:\IF^{\ell \times\ell \times 2}) \otimes_\IF\oIF$.
There is an over-action of $Z\times\SL_2(\IF)$ on
$\Sym^{\ell m} \:\IF^{\ell \times\ell \times 2}$, see \S\ref{subsec:overact}.
By \eqref{eq:Vstardual} we have
\begin{equation}
\label{eq:afterdualizing}
((\Sym^{\ell m} \:\IF^{\ell \times\ell \times 2})^Z)^\star_\IF \ \cong \ 
((\Sym^{\ell m} \:\IF^{\ell \times\ell \times 2})^\star_\IF)_Z.
\end{equation}
 
We write $\IF^{\ell\star} := (\IF^{\ell})_\IF^\star$
and
$\oIF^{\ell\star} := (\oIF^{\ell})_\oIF^\star$
and
$\IF^{\ell\star \times\ell\star \times 2\star} :=
\IF^{\ell\star}\otimes \IF^{\ell\star}\otimes\IF^{2\star}$
and
$\oIF^{\ell\star \times\ell\star \times 2\star} := 
\oIF^{\ell\star}\otimes \oIF^{\ell\star}\otimes\oIF^{2\star}$.
\begin{lemma}\label{lem:dualisomlem}
\[\big(\Sym^{\ell m} \:\IF^{\ell \times\ell \times 2}\big)^{\SL_\ell(\oIF)\times \SL_\ell(\oIF)} \ \cong  \ \big(\Sym^{\ell m} \:\IF^{\ell\star \times\ell\star \times 2\star}\big)^{\SL_\ell(\oIF)\times \SL_\ell(\oIF)}
\]
as $\SL_2(\IF)$-representations.
\end{lemma}
\begin{proof}
Fix an $\IF$-basis of $\IF^{\ell}$.
This is also an $\oIF$-basis of $\oIF^{\ell}$.
This induces a group automorphism of $\SL_\ell(\oIF)$ given by $g \mapsto g^{-T}$.
Let $V$ be any $\SL_\ell(\oIF)$-representation, and let $\widetilde V$ be the $\SL_\ell(\oIF)$-representation that arises from $V$ by composing the action of $\SL_\ell(\oIF)$ on $V$ with the group automorphism $g\mapsto g^{-T}$. It is crucial to note that $V$ and $\widetilde V$ have the same invariant space:
\begin{equation}\label{eq:VVtilde}
V^{\SL_\ell(\oIF)} = \widetilde V^{\SL_\ell(\oIF)},
\end{equation}
because $\SL_\ell(\oIF) = \{g^{-T}\mid g \in \SL_\ell(\oIF)\}$ are the same sets.
For dual spaces, we use the notation $V^{\ell\circ} := \widetilde{V^{\ell\star}}$ as in \cite[\S2.8]{MW:22}.
Using the fixed basis of $\IF^\ell$, let $\star:\IF^{\ell}\to\IF^{\ell\circ}$ be the map that sends a vector to its dual vector.
The map $\star$ has a scalar extension $\oIF^{\ell}\to\oIF^{\ell\circ}$ that we also call $\star$.
This scalar extension is an $\SL_\ell(\oIF)$-isomorphism.
Let $\varphi:\IF^2\to\IF^{2\star}$ be an $\SL_2(\IF)$-isomorphism.
This induces a scalar extension that we also call $\varphi$.
Together, these give rise to an $\SL_\ell(\oIF)\times \SL_\ell(\oIF)\times \SL_2(\oIF)$-isomorphism
\[
\star\otimes\star\otimes\varphi : \oIF^{\ell\times\ell\times 2}
 \ \to \ 
\oIF^{\ell\circ\times\ell\circ\times 2\star}.
\]
This induces an $\SL_\ell(\oIF)\times \SL_\ell(\oIF)\times \SL_2(\oIF)$-isomorphism $(\star\otimes\star\otimes\varphi)^{\odot\ell m} : \Sym^{\ell m}\:\oIF^{\ell\times\ell\times2}
\to
\Sym^{\ell m}\:\oIF^{\ell\circ\times\ell\circ\times 2\star}$.
We take $\SL_\ell(\oIF)\times \SL_\ell(\oIF)$-invariants on both sides, and we use \eqref{eq:VVtilde} on the right-hand side to obtain the $\SL_2(\oIF)$-isomorphism
$(\Sym^{\ell m}\:\oIF^{\ell\times\ell\times2})^{\SL_\ell(\oIF)\times\SL_\ell(\oIF)}
\to
(\Sym^{\ell m}\:\oIF^{\ell\star\times\ell\star\times 2\star})^{\SL_\ell(\oIF)\times\SL_\ell(\oIF)}$
as the restriction of $(\star\otimes\star\otimes\varphi)^{\odot\ell m}$.
This finishes the proof, because $(\star\otimes\star\otimes\varphi)^{\odot\ell m}$
maps elements in
$\Sym^{\ell m}\:\IF^{\ell\times\ell\times2}$
to elements in 
$\Sym^{\ell m}\:\IF^{\ell\star\times\ell\star\times 2\star}$.
\end{proof}
\begin{corollary}
\label{cor:innerdual}
\[
\Big(\big(\Sym^{\ell m} \:\IF^{\ell \times\ell \times 2}\big)^\star_\IF\Big)_{\SL_\ell(\oIF)\times \SL_\ell(\oIF)} \ \cong  \ \Big(\Sym_{\ell m} \:\IF^{\ell \times\ell \times 2}\Big)_{\SL_\ell(\oIF)\times \SL_\ell(\oIF)}
\]
as $\SL_2(\IF)$-representations.
\end{corollary}
\begin{proof}
The dual of the left-hand side of Corollary~\ref{cor:innerdual} is
\[
\left(
\Big(\big(\Sym^{\ell m} \:\IF^{\ell \times\ell \times 2}\big)^\star_\IF\Big)_{\SL_\ell(\oIF)\times \SL_\ell(\oIF)}
\right)_\IF^\ast
\cong
\left(
\Big(\big(\Sym^{\ell m} \:\IF^{\ell \times\ell \times 2}\big)^\star_\IF\Big)_\IF^\ast
\right)^{\SL_\ell(\oIF)\times \SL_\ell(\oIF)}
\cong
\big(\Sym^{\ell m} \:\IF^{\ell \times\ell \times 2}\big)^{\SL_\ell(\oIF)\times \SL_\ell(\oIF)}.
\]
The dual of the right-hand side of Corollary~\ref{cor:innerdual} is
\[
\left(
\Big(\Sym_{\ell m} \:\IF^{\ell \times\ell \times 2}\Big)_{\SL_\ell(\oIF)\times \SL_\ell(\oIF)}
\right)^\star_\IF
\cong
\left(
\Big(\Sym_{\ell m} \:\IF^{\ell \times\ell \times 2}\Big)^\star_\IF\right)^{\SL_\ell(\oIF)\times \SL_\ell(\oIF)}
\cong
\Big(\Sym^{\ell m} \:\IF^{\ell\star \times\ell\star \times 2\star}\Big)^{\SL_\ell(\oIF)\times \SL_\ell(\oIF)}.
\]
Both are isomorphic $\SL_2(\IF)$-representations by Lemma~\ref{lem:dualisomlem}, hence their duals are also isomorphic $\SL_2(\IF)$-representations.
\end{proof}
We prove \eqref{eq:coinvariantsdual} by composing the $\SL_2(\IF)$-isomorphisms:
\begin{eqnarray}
\left(\big(\Sym^{\ell m}\:\IF^{\ell\times\ell\times 2}\big)^{\SL_\ell(\oIF)\times \SL_\ell(\oIF)}\right)^\star_\IF
&\stackrel{\eqref{eq:afterdualizing}}{\cong}&
\left(\big(\Sym^{\ell m}\:\IF^{\ell\times\ell\times 2}\big)^\star_\IF\right)_{\SL_\ell(\oIF)\times \SL_\ell(\oIF)}
\nonumber\\
&\stackrel{\textup{Cor.}~\ref{cor:innerdual}}{\cong}&
\left(\Sym_{\ell m}\:\IF^{\ell\times\ell\times 2}\right)_{\SL_\ell(\oIF)\times \SL_\ell(\oIF)}.
\label{eq:coinvariants}\end{eqnarray}

\section{Hermite Reciprocity, Wronskian, Hodge Isomorphism}
\label{sec:isos}
In this section we discuss the other isomorphisms in Figure~\ref{fig:square} over the standard basis.
We give another proof that the Hermite reciprocity map $\Hermite_{m,\ell} := \Wronski_{\ell,m}^\star\circ \Hodge_{m,\ell} \circ \Wronski_{m,\ell}$ is an isomorphism by combinatorially proving that its matrix is triangular over the standard basis.
Working over the standard basis enables us to consider the behavior of the ``leading coefficients''.
The core argument of our proof is a combinatorial statement about partitions,
see Proposition~\ref{pro:transposition}.
\cite[Exa.~5.1]{MW:22} for a finite example.

Recall that $\mathscr P_k(m,\ell) = \{\la \vdash k, \ \la\subseteq (\ell^m)\}$.
We define $\mathscr P'_k(m,\ell) = \{\la \vdash k, \ \la\subseteq (\ell^m), \ \textup{all $\la_i$ distinct for $i\in\{1,\ldots,m\}$}\}$. Partitions in $\mathscr P'_k(m,\ell)$ are called \emph{regular partitions}.
We make the following definitions for $\la\in\{0,\ldots,\ell\}^m$.
\begin{eqnarray*}
F_{\otimes,\s}(\la) &:=& x^{\la_1}y^{\ell-\la_1}\otimes \cdots \otimes x^{\la_m}y^{\ell-\la_m} \in \tensor^m \Sym^\ell \IF^2
\\
F_{\d,\s}(\lambda) &:=& \sum_{\sigma\in \aS_m/\stab(\la)} F_{\otimes,\s}(\sigma\cdot \la) \in \Sym_m \Sym^\ell \IF^2
\\
F_{\wedge,\s}(\lambda) &:=& x^{\la_1}y^{\ell-\la_1}\wedge \cdots \wedge x^{\la_m}y^{\ell-\la_m} \in \Wedge^m\Sym^\ell \IF^2
\\
F_{\wedge,\d}(\lambda) &:=&
F(\la_1) \wedge F(\la_2) \wedge \cdots \wedge F(\la_m)
\in \Wedge^m\Sym_\ell \IF^2
\end{eqnarray*}

These vectors to partitions $\la$ (regular partitions in the last two cases) form a basis of their respective vector spaces, which we call the standard bases.
The support $\supp(v)$ of a vector $v$ is the set of partitions for which $v$ has a nonzero coefficient.
For two partitions, $\la$ and $\mu$, we say that $\la$ \emph{dominates} $\mu$ if $\forall i:\sum_{j=1}^i\la_j\geq\sum_{j=1}^i\mu_j$.
Let $d_m=(m-1,m-2,\dots,0)$ be the staircase vector of length $m$.
The following linear maps are $\SL_2(\IF)$-equivariant, see for example \cite{MW:22}.

\paragraph{Wronskian:}The Wronskian map $\Wronski_{m,\ell} : \Sym_m\Sym^\ell \IF^2 \to\Wedge^m \Sym^{\ell+m-1} \IF^2$ is defined as
\[
F_{\d,\s}(\la) \mapsto \sum_{\sigma\in \aS_m/\stab (\la)} F_{\wedge,\s} (\sigma\la + d_m).
\]
It is easy to see that $\supp(\Wronski_{m,\ell}(F_{\d,\s}(\la)))$ is a poset with respect to the dominance order with maximum element $\la+d_m$.

\paragraph{Hodge:}
The Hodge map $\Hodge_{m,\ell} : \Wedge^m \Sym^{\ell+m-1} \IF^2 \to \Wedge^\ell \Sym_{\ell+m-1} \IF^2$ is defined as
\[
F_{\wedge,\s}(\la) \mapsto F_{\wedge,\d}(((\ell+m-1)^\ell)-\la^\c),
\]
where $\la^\c$ is the partition that has a row of length $i$ if and only if $\la$ does not have a row of length $i$.

\paragraph{Dual Wronskian:}
The dual Wronskian map
$\Wronski_{\ell,m}^\star: \Wedge^\ell \Sym_{\ell+m-1} \IF^2 \to \Sym^\ell\Sym_m \IF^2$
is defined as
\[
F_{\wedge,\d}(\la) \mapsto \sum_{\pi\in \aS_\ell} F_{\s,\d}(\pi\la - d_\ell)
\]
where we set $F_{\s,\d}(\nu)=0$ if for some $i$ we have $\nu_i<0$ or $\nu_i>m$.
It is easy to see that $\supp(\Wronski^\star_{\ell,m}(F_{\wedge,\d}(\la)))$ is a poset with respect to the dominance order with minimum element $\la-d_\ell$.

We now describe a combinatorial counterpart to the top square in the top diagram of Figure~\ref{fig:square},
which we will use in the proof of Theorem~\ref{thm:hermiterec} below.
We define the following partition transformations (see also Figure \ref{fig:partitions}:
\begin{figure}
    \centering
\begin{tikzpicture}[xscale=1.5,yscale=1]
\begin{scope}[yscale=2,xscale=1.5]
\node (ll) at (0,0) {$\ytableausetup{smalltableaux}
\begin{ytableau}
 *(gray) & *(gray) & *(gray) & *(gray) \\
 *(gray) & *(gray) &*(white) &*(white) \\
 *(gray) & *(gray) &*(white) &*(white) \\
 *(gray) &*(white) &*(white) &*(white) \\
 *(white)&*(white) &*(white) &*(white) \\
\end{ytableau}$};
\node (ul) at (0,2) {$\ytableausetup{smalltableaux}
\begin{ytableau}
 *(gray) & *(gray) & *(gray) & *(gray) & *(gray) & *(gray) & *(gray) & *(gray) \\
 *(gray) & *(gray) & *(gray) & *(gray) & *(gray) &*(white) &*(white) &*(white) \\
 *(gray) & *(gray) & *(gray) & *(gray) &*(white) &*(white) &*(white) &*(white) \\
 *(gray) & *(gray) &*(white) &*(white) &*(white) &*(white) &*(white) &*(white) \\
 *(white)&*(white) &*(white) &*(white) &*(white) &*(white) &*(white) &*(white) \\
\end{ytableau}$};
\node (lr) at (3,0) {$\begin{ytableau}
*(gray) & *(gray) & *(gray) & *(gray) &*(white) \\
*(gray) & *(gray) & *(gray) &*(white) &*(white)\\
*(gray) &*(white) &*(white) &*(white) &*(white)\\
 *(gray) &*(white) &*(white) &*(white) &*(white)\\
\end{ytableau}$ 
};
\node (ur) at (3,2) { 
$\ytableausetup{smalltableaux}
\begin{ytableau}
 *(gray) & *(gray) & *(gray) & *(gray) & *(gray) & *(gray) & *(gray) &*(white) \\
 *(gray) & *(gray) & *(gray) & *(gray) & *(gray) &*(white) &*(white) &*(white) \\
 *(gray) & *(gray) &*(white) &*(white) &*(white) &*(white) &*(white) &*(white) \\
 *(gray) &*(white) &*(white) &*(white) &*(white) &*(white) &*(white) &*(white) \\
\end{ytableau}$ 
};
\draw[-Triangle] (ll) -- (ul) node [midway, right, fill=white] { 
$\ytableausetup{smalltableaux}
\begin{ytableau}
 *(gray)\bullet & *(gray)\bullet & *(gray)\bullet & *(gray)\bullet & *(gray) & *(gray) & *(gray) & *(gray) \\
 *(gray)\bullet & *(gray)\bullet & *(gray)\bullet & *(gray) & *(gray) &*(white) &*(white) &*(white) \\
 *(gray)\bullet & *(gray)\bullet & *(gray) & *(gray) &*(white) &*(white) &*(white) &*(white) \\
 *(gray)\bullet & *(gray) &*(white) &*(white) &*(white) &*(white) &*(white) &*(white) \\
 *(white)&*(white) &*(white) &*(white) &*(white) &*(white) &*(white) &*(white) \\
\end{ytableau}$ 
};
\draw[-Triangle] (ul) -- (ur) node [midway, above, fill=white] { 
\begin{tikzpicture}[scale=0.3]\ytableausetup{smalltableaux}
\node at (0,0){\begin{ytableau}
 *(gray) & *(gray) & *(gray) & *(gray) & *(gray) & *(gray) & *(gray) & *(gray) \\
*(white) &*(white) &*(white) &*(white) &*(white) &*(white) &*(white) & *(gray) \\
*(white) &*(white) &*(white) &*(white) &*(white) &*(white) & *(gray) & *(gray) \\
 *(gray) & *(gray) & *(gray) & *(gray) & *(gray) &*(white) &*(white) &*(white) \\
 *(gray) & *(gray) & *(gray) & *(gray) &*(white) &*(white) &*(white) &*(white) \\
*(white) &*(white) &*(white) & *(gray) & *(gray) & *(gray) & *(gray) & *(gray) \\
 *(gray) & *(gray) &*(white) &*(white) &*(white) &*(white) &*(white) &*(white) \\
*(white) & *(gray) & *(gray) & *(gray) & *(gray) & *(gray) & *(gray) & *(gray) \\
 *(white)&*(white) &*(white) &*(white) &*(white) &*(white) &*(white) &*(white) \\
\end{ytableau}}; 
\draw[ultra thick, -] (-4,0.4) -- ++(0,1) -- ++(1,0) -- ++(0,1) -- ++(1,0) -- ++(0,1) -- ++(1,0) -- ++(0,1) -- ++(1,0) -- ++(0,1) -- ++(1,0) -- ++(0,1) -- ++(1,0) -- ++(0,1) -- ++(1,0) -- ++(0,1) -- ++(1,0) -- ++(0,1);
\end{tikzpicture} 
};
\draw[-Triangle] (ur) -- (lr) node [midway, right, fill=white] { 
$\ytableausetup{smalltableaux}
\begin{ytableau}
 *(gray)\bullet & *(gray)\bullet & *(gray)\bullet & *(gray) & *(gray) & *(gray) & *(gray) &*(white) \\
 *(gray)\bullet & *(gray)\bullet & *(gray) & *(gray) & *(gray) &*(white) &*(white) &*(white) \\
 *(gray)\bullet & *(gray) &*(white) &*(white) &*(white) &*(white) &*(white) &*(white) \\
 *(gray) &*(white) &*(white) &*(white) &*(white) &*(white) &*(white) &*(white) \\
\end{ytableau}$ 
};
\draw[-Triangle] (ll) -- (ul) node [midway, left, fill=white] {$\widetilde{\Wronski}_{m,\ell}$};
\draw[-Triangle] (ul) -- (ur) node [midway, below, fill=white] {$\widetilde{\Hodge}_{m,\ell}$};
\draw[-Triangle] (ur) -- (lr) node [midway, left, fill=white] {$\widetilde{\Wronski}^\star_{\ell,m}$};
\draw[-Triangle] (ll) -- (lr) node [midway, above, fill=white] {$\widetilde{\Hermite}_{m,\ell}$};
\end{scope}
\end{tikzpicture}
\caption{An illustration of Proposition~\ref{pro:transposition}. Here $m=5$, $\ell=4$, $\la=(4,2,2,1)$.}
\label{fig:partitions}
\end{figure}
\begin{itemize}
\item
$\widetilde{\Wronski}_{m,\ell}:\mathscr P(m,\ell)\to \mathscr P'(m,\ell + m - 1)$, \, 
$\widetilde{\Wronski}_{m,\ell}(\la)=\la+d_m$.
\item
$\widetilde{\Hodge}_{m,\ell}:\mathscr P'(m,\ell+m-1)\to\mathscr P'(\ell,\ell+m-1)$, \, 
$\widetilde{\Hodge}_{m,\ell}(\la) = ((\ell+m-1)^m)-\la^\c$.
\item $\widetilde\Wronski^\star_{\ell,m}:\mathscr P'(\ell,\ell+m-1) \to \mathscr P(\ell,m)$, \,
$\widetilde{\Wronski}^\star_{\ell,m}(\la) = \la-d_\ell$.
\item $\widetilde\Hermite_{m,\ell}:\mathscr P(m,\ell)\to\mathscr P(\ell,m)$, \,
$\widetilde\Hermite_{m,\ell}(\la)=\la^T$.
\end{itemize}

\begin{proposition}
\label{pro:transposition}
The following diagram commutes.
\begin{center}
{\centering
\begin{tikzpicture}
\node (ul) at (0,0) {$\mathscr P'_{k+\binom{m}{2}}(m,\ell+m-1)$};
\node (ur) at (6,0) {$\mathscr P'_{k+\binom{\ell}{2}}(\ell,\ell+m-1)$};
\node (ll) at (0,-2) {$\mathscr P_k(m,\ell)$};
\node (lr) at (6,-2) {$\mathscr P_k(\ell,m)$};
\draw[-Triangle] (ll) -- (ul) node [midway, left, fill=white] {$\widetilde{\Wronski}_{m,\ell}$};
\draw[-Triangle] (ul) -- (ur) node [midway, below, fill=white] {$\widetilde{\Hodge}_{m,\ell}$};
\draw[-Triangle] (ur) -- (lr) node [midway, right, fill=white] {$\widetilde{\Wronski}^\star_{\ell,m}$};
\draw[-Triangle] (ll) -- (lr) node [midway, above, fill=white] {$\widetilde{\Hermite}_{m,\ell}$};
\end{tikzpicture}
}
\end{center}
\end{proposition}

\begin{proof}
Let $\la\in\mathscr P_k(m,\ell)$.
We have $\widetilde{\Wronski}_{m,\ell}(\la) = \la+d_m$,
and $\big(\widetilde{\Wronski}^\star_{\ell,m}\big)^{-1}(\la^T) = \la^T+d_\ell$.
Since these two partitions together have $m+\ell$ many rows,
to finish the proof
it suffices to prove that
$\la+d_m$ and $((\ell+m-1)^\ell)-(\la^T+d_\ell)$ have no row length in common.
This is proved for example in \cite[eq.~(1.7)]{MacD98}.
For being self-contained, we provide a short proof here.
For the sake of contradiction, assume there is a common element.
That means there are indices $i\in \{1,\dots, m\}$ and $j\in \{1,\dots, \ell\}$, such that $\la_i + m-i = (\ell+m -1)- (\la^T_j + \ell - j)$, which is equivalent to $\la_i+\la^T_j=i+j-1$.
We make a case distinction and rule out both cases.
In the case that $\la_i\geq j$,
the Young diagram of $\la$ has a box in row $i$ and column $j$, hence $\la^T_j\geq i$, thus we get $\la_i+\la^T_j\geq i+j>i+j-1$.
Otherwise, in the case $\la_i \leq j-1$, the Young diagram of $\la$ has no box in row $i$ and column $j$,
and hence $\la^T_j\leq i-1$, so we get $\la_i+\la^T_j\leq i+j-2 < i+j-1$.
In both cases we reached the desired contradiction $\la_i+\la^T_j\neq i+j-1$.
\end{proof}

\begin{claim}\label{cla:Hodgereverse}
$\widetilde{\Hodge}_{m,\ell}$ reverses the dominance order, i.e., is antimonotone.
\end{claim}
\begin{proof}
The isomorphisms $\widetilde{\Wronski}_{m,\ell}$ and $\widetilde{\Wronski}^\star_{\ell,m}$ respect the dominance order, while transposition is known to be anti-isomorphism, so it reverses the order.
By the commutativity of the diagram from Proposition~\ref{pro:transposition}, we get that $\widetilde{\Hodge}_{m,\ell}$ reverses the dominance order, as well.
\end{proof}

\begin{theorem}[Hermite reciprocity]\label{thm:hermiterec}
The $\GL_2(\IF)$-equivariant map $\Hermite_{m,\ell} =  \Wronski_{\ell,m}^\star\circ \Hodge_{m,\ell} \circ \Wronski_{m,\ell}$ is an isomorphism.
\end{theorem}

\begin{proof}
Let $\la\in\mathscr P(m,\ell)$.
Then $\supp(\Wronski_{m,\ell}(\la))$ is a poset with respect to the dominance order whose maximum element is $\widetilde{\Wronski}_{m,\ell}(\la)$.
Hence, $\supp(\Hodge_{m,\ell}(\Wronski_{m,\ell}(\la)))$ is a poset with respect to the dominance order whose minimum element is $\widetilde{\Hodge}_{m,\ell}(\widetilde{\Wronski}_{m,\ell}(\la))$, see Claim~\ref{cla:Hodgereverse}.
Therefore, $\supp(\Wronski^\star_{\ell,m}(\Hodge_{m,\ell}(\Wronski_{m,\ell}(\la))))$ is a poset with respect to the dominance order whose minimum element is $\widetilde{\Wronski}^\star_{\ell,m}(\widetilde{\Hodge}_{m,\ell}(\widetilde{\Wronski}_{m,\ell}(\la)))$.
We conclude that the matrix of $\Wronski^\star_{\ell,m}\circ\Hodge_{m,\ell}\circ\Wronski_{m,\ell}$
is triangular with respect to the standard basis.
Proposition~\ref{pro:transposition} implies that the diagonal elements are all 1, which finishes the proof.
\end{proof}

\section{Some computations over finite fields}
\label{sec:computations}
We determine via computer calculations that
$(\Sym^{\bullet}(\IF^{\ell\times\ell\times2}))^{\SL_\ell(\IF)\times \SL_\ell(\IF)}$
can be a strict superset of
$(\Sym^{\bullet}(\IF^{\ell\times\ell\times2}))^{\SL_\ell(\oIF)\times \SL_\ell(\oIF)}$.
We do this by finding invariant polynomials $p(A)\in (\Sym^d(\IF^{\ell\times\ell}))^{\SL_\ell(\IF)\times \SL_\ell(\IF)}$ over small finite fields $\IF$, which cannot be written as a power of the determinant. These can be transformed into invariant polynomials in $(\Sym^d(\IF^{\ell\times\ell\times2}))^{\SL_\ell(\IF)\times \SL_\ell(\IF)}$ as usual via $p(A)\boxtimes F(k)$ for any $0\leq k\leq \ell$, but these invariants cannot be generated by the polynomials $M_\ell(k)$.
We present here a selection of these invariants.

\begin{itemize}
    \item $(\Sym^3(\IF_2^{2\times2}))^{\SL_2(\IF_2)\times \SL_2(\IF_2)}:$\\
    $A_{2,2} A_{1,1}^2+A_{2,2}^2 A_{1,1}+A_{1,2} A_{2,1}^2+A_{1,2}^2 A_{2,1}$

	\item $(\Sym^4(\IF_2^{2\times2}))^{\SL_2(\IF_2)\times \SL_2(\IF_2)}:$\\
    $A_{1,1}^4+A_{2,2} A_{1,1}^3+A_{1,2}^2 A_{1,1}^2+A_{2,1}^2 A_{1,1}^2+A_{1,2} A_{2,1} A_{1,1}^2+A_{1,2} A_{2,2} A_{1,1}^2+A_{2,1} A_{2,2} A_{1,1}^2+A_{2,2}^3A_{1,1}+A_{1,2} A_{2,1}^2 A_{1,1}+A_{1,2} A_{2,2}^2 A_{1,1}+A_{2,1} A_{2,2}^2A_{1,1}+A_{1,2}^2 A_{2,1} A_{1,1}+A_{1,2}^2 A_{2,2} A_{1,1}+A_{2,1}^2 A_{2,2} A_{1,1}+A_{1,2} A_{2,1} A_{2,2} A_{1,1}+A_{1,2}^4+A_{2,1}^4+A_{2,2}^4+A_{1,2}A_{2,1}^3+A_{1,2}^2 A_{2,2}^2+A_{2,1}^2 A_{2,2}^2+A_{1,2} A_{2,1} A_{2,2}^2+A_{1,2}^3A_{2,1}+A_{1,2} A_{2,1}^2 A_{2,2}+A_{1,2}^2 A_{2,1} A_{2,2}$

	\item $(\Sym^4(\IF_3^{2\times2}))^{\SL_2(\IF_3)\times \SL_2(\IF_3)}:$\\$A_{1,2}^2 A_{2,1}^2+A_{1,1} A_{1,2} A_{2,2} A_{2,1}+A_{1,1}^2 A_{2,2}^2$

    \item $(\Sym^5(\IF_2^{2\times2}))^{\SL_2(\IF_2)\times \SL_2(\IF_2)}:$
    \begin{itemize}
	\item[$\bullet$] $A_{2,2}^2 A_{1,1}^3+A_{2,2}^3 A_{1,1}^2+A_{1,2} A_{2,1} A_{2,2} A_{1,1}^2+A_{1,2} A_{2,1} A_{2,2}^2 A_{1,1}+A_{1,2} A_{2,1}^2 A_{2,2} A_{1,1}+A_{1,2}^2 A_{2,1} A_{2,2} A_{1,1}+A_{1,2}^2A_{2,1}^3+A_{1,2}^3 A_{2,1}^2$
	\item[$\bullet$] $A_{2,2} A_{1,1}^4+A_{2,2}^4 A_{1,1}+A_{1,2} A_{2,1}^4+A_{1,2}^4 A_{2,1}$
   \end{itemize}
   
	\item $(\Sym^6(\IF_2^{2\times2}))^{\SL_2(\IF_2)\times \SL_2(\IF_2)}:$\begin{itemize}
	\item[$\bullet$] $A_{2,2} A_{1,1}^5+A_{1,2}^2 A_{1,1}^4+A_{1,2} A_{2,1} A_{1,1}^4+A_{1,2} A_{2,2} A_{1,1}^4+A_{2,1} A_{2,2}^2 A_{1,1}^3+A_{2,1}^2 A_{2,2} A_{1,1}^3+A_{1,2}^4 A_{1,1}^2+A_{1,2} A_{2,1}^3A_{1,1}^2+A_{1,2} A_{2,2}^3 A_{1,1}^2+A_{1,2} A_{2,1}^2 A_{2,2} A_{1,1}^2+A_{2,2}^5 A_{1,1}+A_{2,1} A_{2,2}^4 A_{1,1}+A_{1,2}^2 A_{2,1}^3 A_{1,1}+A_{1,2}^2 A_{2,2}^3 A_{1,1}+A_{1,2}^2 A_{2,1}A_{2,2}^2 A_{1,1}+A_{1,2}^4 A_{2,1} A_{1,1}+A_{1,2}^4 A_{2,2} A_{1,1}+A_{2,1}^4 A_{2,2} A_{1,1}+A_{1,2} A_{2,1}^5+A_{2,1}^2 A_{2,2}^4+A_{1,2} A_{2,1} A_{2,2}^4+A_{2,1}^4 A_{2,2}^2+A_{1,2}^3A_{2,1} A_{2,2}^2+A_{1,2}^5 A_{2,1}+A_{1,2} A_{2,1}^4 A_{2,2}+A_{1,2}^3 A_{2,1}^2 A_{2,2}$

	\item[$\bullet$] $A_{2,2} A_{1,1}^5+A_{2,1}^2 A_{1,1}^4+A_{1,2} A_{2,1} A_{1,1}^4+A_{2,1} A_{2,2} A_{1,1}^4+A_{1,2} A_{2,2}^2 A_{1,1}^3+A_{1,2}^2 A_{2,2} A_{1,1}^3+A_{2,1}^4 A_{1,1}^2+A_{2,1} A_{2,2}^3A_{1,1}^2+A_{1,2}^3 A_{2,1} A_{1,1}^2+A_{1,2}^2 A_{2,1} A_{2,2} A_{1,1}^2+A_{2,2}^5 A_{1,1}+A_{1,2} A_{2,1}^4 A_{1,1}+A_{1,2} A_{2,2}^4 A_{1,1}+A_{2,1}^2 A_{2,2}^3 A_{1,1}+A_{1,2}^3 A_{2,1}^2A_{1,1}+A_{1,2} A_{2,1}^2 A_{2,2}^2 A_{1,1}+A_{1,2}^4 A_{2,2} A_{1,1}+A_{2,1}^4 A_{2,2} A_{1,1}+A_{1,2} A_{2,1}^5+A_{1,2}^2 A_{2,2}^4+A_{1,2} A_{2,1} A_{2,2}^4+A_{1,2}^4 A_{2,2}^2+A_{1,2}A_{2,1}^3 A_{2,2}^2+A_{1,2}^5 A_{2,1}+A_{1,2}^2 A_{2,1}^3 A_{2,2}+A_{1,2}^4 A_{2,1} A_{2,2}$

	\item[$\bullet$] $A_{2,2}^2 A_{1,1}^4+A_{2,2}^4 A_{1,1}^2+A_{1,2}^2 A_{2,1}^4+A_{1,2}^4 A_{2,1}^2$

	\item[$\bullet$] $A_{2,1}^2 A_{1,1}^4+A_{2,1} A_{2,2} A_{1,1}^4+A_{2,2}^3 A_{1,1}^3+A_{2,1} A_{2,2}^2 A_{1,1}^3+A_{2,1}^2 A_{2,2} A_{1,1}^3+A_{2,1}^4 A_{1,1}^2+A_{1,2} A_{2,1}^3 A_{1,1}^2+A_{1,2} A_{2,2}^3A_{1,1}^2+A_{1,2}^2 A_{2,1}^2 A_{1,1}^2+A_{1,2}^2 A_{2,2}^2 A_{1,1}^2+A_{2,1}^2 A_{2,2}^2 A_{1,1}^2+A_{1,2}^2 A_{2,1} A_{2,2} A_{1,1}^2+A_{1,2} A_{2,1}^4 A_{1,1}+A_{1,2} A_{2,2}^4A_{1,1}+A_{1,2}^2 A_{2,1}^3 A_{1,1}+A_{1,2}^2 A_{2,2}^3 A_{1,1}+A_{1,2} A_{2,1}^2 A_{2,2}^2 A_{1,1}+A_{1,2}^2 A_{2,2}^4+A_{1,2}^3 A_{2,1}^3+A_{1,2}^4 A_{2,2}^2+A_{1,2}^2 A_{2,1}^2A_{2,2}^2+A_{1,2}^3 A_{2,1} A_{2,2}^2+A_{1,2}^3 A_{2,1}^2 A_{2,2}+A_{1,2}^4 A_{2,1} A_{2,2}$
    \end{itemize}

	\item $(\Sym^6(\IF_3^{2\times2}))^{\SL_2(\IF_3)\times \SL_2(\IF_3)}:$\\$A_{2,2}^2 A_{1,1}^4+2 A_{1,2} A_{2,1} A_{2,2} A_{1,1}^3+A_{2,2}^4 A_{1,1}^2+2 A_{1,2} A_{2,1} A_{2,2}^3 A_{1,1}+2 A_{1,2} A_{2,1}^3 A_{2,2} A_{1,1}+2 A_{1,2}^3 A_{2,1} A_{2,2} A_{1,1}+A_{1,2}^2A_{2,1}^4+A_{1,2}^4 A_{2,1}^2$

\end{itemize}

\section*{Acknowledgments}
The authors are grateful to Mark Wildon and Darij Grinberg for comments on an earlier version of this paper.
Grinberg provided a shorter proof of Claim~\ref{cla:Hodgereverse}, which is the one presented in this version, and pointed us to \cite[eq.~(1.7)]{MacD98} to shorten the proof of Proposition~\ref{pro:transposition}.
Moreover, he showed us a different definition of $\boxtimes$, which we use in this version, and his comments led us to develop the generalization of the coinvariant space that is introduced in \S\ref{sec:coinvariants}.

\bibliographystyle{alpha}
\bibliography{lit}

\end{document}